\documentclass[submission,copyright,creativecommons]{eptcs}
\usepackage{amsthm,amsmath,amssymb}
\usepackage{tikz}

\pgfdeclarelayer{foreground}
\pgfdeclarelayer{background}
\pgfdeclarelayer{morphismlayer}
\pgfsetlayers{background,main,morphismlayer,foreground}

\usetikzlibrary{matrix}
\usetikzlibrary{decorations.markings}
\usetikzlibrary{shapes.geometric}

\tikzstyle{dot}=[circle, draw=black, fill=black!20, inner sep=.4ex, node on layer=foreground]

\tikzstyle{whitedot}=[circle, draw=black, fill=white, inner sep=.4ex, node on layer=foreground]
\tikzstyle{greydot}=[circle, draw=black, fill=black!20, inner sep=.4ex, node on layer=foreground]
\tikzstyle{darkgreydot}=[circle, draw=black, fill=black!50, inner sep=.4ex, node on layer=foreground]
\tikzstyle{blackdot}=[circle, draw=black, fill=black, inner sep=.4ex, node on layer=foreground]

\tikzstyle{triangle} = [regular polygon, regular polygon sides=3, draw=black, fill=black!20,scale=0.3, node on layer=foreground]

\tikzstyle{whitetriangle}=[triangle, fill=white]
\tikzstyle{greytriangle}=[triangle, fill=black!20]
\tikzstyle{darkgreytriangle}=[triangle, fill=black!50]
\tikzstyle{blacktriangle}=[triangle, fill=black]

\tikzstyle{invertedtriangle} = [triangle,scale=-1]
\tikzstyle{whiteinvertedtriangle}=[invertedtriangle, fill=white]
\tikzstyle{greyinvertedtriangle}=[invertedtriangle, fill=black!20]
\tikzstyle{darkgreyinvertedtriangle}=[invertedtriangle, fill=black!50]
\tikzstyle{blackinvertedtriangle}=[invertedtriangle, fill=black]

\tikzstyle{triangletwo} = [semicircle, draw=black, fill=black!20,scale=0.4, node on layer=foreground]

\tikzstyle{whitetriangletwo}=[triangletwo, fill=white]
\tikzstyle{greytriangletwo}=[triangletwo, fill=black!20]
\tikzstyle{darkgreytriangletwo}=[triangletwo, fill=black!50]
\tikzstyle{blacktriangletwo}=[triangletwo, fill=black]

\tikzstyle{invertedtriangletwo} = [triangletwo,scale=-1]
\tikzstyle{whiteinvertedtriangletwo}=[invertedtriangletwo, fill=white]
\tikzstyle{greyinvertedtriangletwo}=[invertedtriangletwo, fill=black!20]
\tikzstyle{darkgreyinvertedtriangletwo}=[invertedtriangletwo, fill=black!50]
\tikzstyle{blackinvertedtriangletwo}=[invertedtriangletwo, fill=black]

\tikzstyle{morphism}=[font=\small,morphismshape]
\tikzstyle{box}=[rectangle,inner sep=.4ex, draw=black, node on layer=foreground]

\tikzstyle{groundtwo} = [semicircle, draw=black, fill=white,scale=0.8, node on layer=foreground]

\newif\ifvflip\pgfkeys{/tikz/vflip/.is if=vflip}
\newif\ifhflip\pgfkeys{/tikz/hflip/.is if=hflip}
\newif\ifhvflip\pgfkeys{/tikz/hvflip/.is if=hvflip}
\newlength\morphismheight
\newlength\wedgewidth
\tikzset{width/.initial=1mm}

\newlength\minimummorphismwidth
\newlength\stateheight
\newlength\minimumstatewidth
\newlength\connectheight
\tikzset{colour/.initial=white}

\tikzstyle{mixed}=[line width=.7pt]
\tikzstyle{pure}=[line width=.7pt]
\tikzset{diredge/.style={decoration={
  markings,
  mark=at position 0.525 with {\arrow{#1}}},postaction={decorate}}}
\tikzset{
    diredge/.default=>
}

\tikzset{diredgestart/.style={decoration={
  markings,
  mark=at position 4pt with {\arrow{#1}}},postaction={decorate}}}
\tikzset{
    diredgestart/.default=<
}

\tikzset{diredgeend/.style={decoration={
  markings,
  mark=at position 1 with {\arrow{#1}}},postaction={decorate}}}
\tikzset{
    diredgeend/.default=>
}

\makeatletter

\pgfdeclareshape{morphismshape}
{
    \savedanchor\centerpoint
    {
        \pgf@x=0pt
        \pgf@y=0pt
    }
    \anchor{center}{\centerpoint}
    \anchorborder{\centerpoint}
    \saveddimen\overallwidth
    {
        \pgfkeysgetvalue{/tikz/width}{\minwidth}
        \pgf@x=\wd\pgfnodeparttextbox
        \ifdim\pgf@x<\minwidth
            \pgf@x=\minwidth
        \fi
    }
    \savedanchor{\upperrightcorner}
    {
        \pgf@y=.5\ht\pgfnodeparttextbox
        \advance\pgf@y by -.5\dp\pgfnodeparttextbox
        \pgf@x=.5\wd\pgfnodeparttextbox
    }
    \anchor{north}
    {
        \pgf@x=0pt
        \pgf@y=0.5\morphismheight
    }
    \anchor{north east}
    {
        \pgf@x=\overallwidth
        \multiply \pgf@x by 4
        \divide \pgf@x by 5
        \pgf@y=0.5\morphismheight
    }
    \anchor{east}
    {
        \pgf@x=\overallwidth
        \divide \pgf@x by 2
        \advance \pgf@x by 5pt
        \pgf@y=0pt
    }
    \anchor{west}
    {
        \pgf@x=-\overallwidth
        \divide \pgf@x by 2
        \advance \pgf@x by -5pt
        \pgf@y=0pt
    }
    \anchor{north west}
    {
        \pgf@x=-\overallwidth
        \multiply \pgf@x by 4
        \divide \pgf@x by 5
        \pgf@y=0.5\morphismheight
    }
    \anchor{south east}
    {
        \pgf@x=\overallwidth
        \multiply \pgf@x by 4
        \divide \pgf@x by 5
        \pgf@y=-0.5\morphismheight
    }
    \anchor{south west}
    {
        \pgf@x=-\overallwidth
       \multiply \pgf@x by 4
        \divide \pgf@x by 5
        \pgf@y=-0.5\morphismheight
    }
    \anchor{south}
    {
        \pgf@x=0pt
        \pgf@y=-0.5\morphismheight
    }
    \anchor{text}
    {
        \upperrightcorner
        \pgf@x=-\pgf@x
        \pgf@y=-\pgf@y
    }
    \backgroundpath
    {
    \begin{scope}
        \pgfkeysgetvalue{/tikz/fill}{\morphismfill}
        \pgfsetstrokecolor{black}
        \begin{scope}
        \pgfsetstrokecolor{black}
        \pgfsetfillcolor{white}
                \ifhflip
                    \pgftransformyscale{-1}
                \fi
                \ifvflip
                    \pgftransformxscale{-1}
                \fi
                \ifhvflip
                    \pgftransformxscale{-1}
                    \pgftransformyscale{-1}
                \fi
                \pgfpathmoveto{\pgfpoint
                    {-0.5*\overallwidth-5pt}
                    {0.5*\morphismheight}}
                \pgfpathlineto{\pgfpoint
                    {0.5*\overallwidth+5pt}
                    {0.5*\morphismheight}}
                \pgfpathlineto{\pgfpoint
                    {0.5*\overallwidth + \wedgewidth}
                    {-0.5*\morphismheight}}
                \pgfpathlineto{\pgfpoint
                    {-0.5*\overallwidth-5pt}
                    {-0.5*\morphismheight}}
                \pgfpathclose
                \pgfusepath{fill,stroke}
        \end{scope}
    \end{scope}
    }
}

\pgfdeclareshape{ground}
{
    \savedanchor\centerpoint
    {
        \pgf@x=0pt
        \pgf@y=0pt
    }
    \anchor{center}{\centerpoint}
    \anchorborder{\centerpoint}

    \anchor{north}
    {
        \pgf@x=0pt
        \pgf@y=0.16\stateheight
    }
    \anchor{south}
    {
        \pgf@x=0pt
        \pgf@y=0pt
    }
    \saveddimen\overallwidth
    {
        \pgfkeysgetvalue{/pgf/minimum width}{\minwidth}
        \pgf@x=\minimumstatewidth
        \ifdim\pgf@x<\minwidth
            \pgf@x=\minwidth
        \fi
    }
    \backgroundpath
    {
        \begin{pgfonlayer}{foreground}
        \pgfsetstrokecolor{black}
        \pgfsetlinewidth{1.25pt}
        \ifhflip
            \pgftransformyscale{-1}
        \fi
        \pgftransformscale{0.5}
        \pgfpathmoveto{\pgfpoint{-0.5*\overallwidth}{0pt}}
        \pgfpathlineto{\pgfpoint{0.5*\overallwidth}{0pt}}
        \pgfpathmoveto{\pgfpoint{-0.33*\overallwidth}{0.33*\stateheight}}
        \pgfpathlineto{\pgfpoint{0.33*\overallwidth}{0.33*\stateheight}}
        \pgfpathmoveto{\pgfpoint{-0.16*\overallwidth}{0.66*\stateheight}}
        \pgfpathlineto{\pgfpoint{0.16*\overallwidth}{0.66*\stateheight}}
        \pgfpathmoveto{\pgfpoint{-0.02*\overallwidth}{\stateheight}}
        \pgfpathlineto{\pgfpoint{0.02*\overallwidth}{\stateheight}}
        \pgfusepath{stroke}
        \end{pgfonlayer}
    }
}

\makeatletter
\pgfkeys{%
  /tikz/on layer/.code={
    \pgfonlayer{#1}\begingroup
    \aftergroup\endpgfonlayer
    \aftergroup\endgroup
  },
  /tikz/node on layer/.code={
    \gdef\node@@on@layer{%
      \setbox\tikz@tempbox=\hbox\bgroup\pgfonlayer{#1}\unhbox\tikz@tempbox\endpgfonlayer\egroup}
    \aftergroup\node@on@layer
  },
  /tikz/end node on layer/.code={
    \endpgfonlayer\endgroup\endgroup
  }
}
\def\node@on@layer{\aftergroup\node@@on@layer}
\makeatother

\makeatother
\newcommand{\tinymult}[1][dot]{
\smash{\raisebox{-2pt}{\hspace{-5pt}\ensuremath{\begin{pic}[scale=0.4,yscale=-1]
    \node (0) at (0,0) {};
    \node[#1, inner sep=1.5pt] (1) at (0,0.55) {};
    \node (2) at (-0.5,1) {};
    \node (3) at (0.5,1) {};
    \draw [pure] (0.center) to (1.center);
    \draw [pure] (1.center) to [out=left, in=down, out looseness=1.5] (2.center);
    \draw [pure] (1.center) to [out=right, in=down, out looseness=1.5] (3.center);
    \pgfsetlinewidth{0.7pt}
    \node[#1, inner sep=1.5pt] (1) at (0,0.55) {};
\end{pic}
}\hspace{-3pt}}}}

\newcommand{\tinyunit}[1][dot]{
\smash{\raisebox{1pt}{\hspace{-3pt}\ensuremath{\begin{pic}[scale=0.4,yscale=-1]
    \node (0) at (0,0) {};
    \node[#1, inner sep=1.5pt] (1) at (0,0.55) {};
    \draw [pure] (0.center) to (1.north);
\end{pic}
}\hspace{-1pt}}}}

\newcommand{\tinyfrob}[1][triangle]{
\smash{\raisebox{0pt}{\hspace{-3pt}\ensuremath{\begin{pic}[scale=0.4]
    \node (0) at (0,0) {};
    \node[#1, inner sep=1.5pt,scale=2] (1) at (0,0.4) {};
    \node (2) at (0,1) {};
    \draw (1.north) to (2.center);
    \draw [pure] (0.center) to (1.south);
\end{pic}
}\hspace{-1pt}}}}

\newcommand{\tinyfrobtwo}[1][triangletwo]{
\smash{\raisebox{0pt}{\hspace{-3pt}\ensuremath{\begin{pic}[scale=0.4]
    \node (0) at (0,0) {};
    \node[#1, inner sep=1.5pt,scale=2] (1) at (0,0.4) {};
    \node (2) at (0,1) {};
    \draw (1.north) to (2.center);
    \draw [pure] (0.center) to (1.south);
\end{pic}
}\hspace{-1pt}}}}

\newcommand{\tinyground}[1][ground]{
\smash{\raisebox{-2pt}{\hspace{-3pt}\ensuremath{\begin{pic}[scale=0.4]
    \node[#1, scale=0.6] (1) at (0,0.4) {};
    \draw [pure] (1.south) to +(0,-.3);
\end{pic}
}\hspace{-1pt}}}}

\newcommand{\tinygroundtwo}[1][groundtwo]{
\smash{\raisebox{-2pt}{\hspace{-3pt}\ensuremath{\begin{pic}[scale=0.4]
    \node[#1, scale=0.6] (1) at (0,0.4) {};
    \draw [pure] (1.south) to +(0,-.3);
\end{pic}
}\hspace{-1pt}}}}

\newcommand{\tinyz}[1][z]{
\smash{\raisebox{-2pt}{\hspace{-3pt}\ensuremath{\begin{pic}
    \node[box] (1) at (0,0) {#1};
\end{pic}
}\hspace{-1pt}}}}

\newcommand{\tinypants}{
\smash{\raisebox{-2pt}{\hspace{-3pt}\ensuremath{\begin{pic}[xscale=0.2,yscale=0.1]
		\node (1) at (-1.25, -1.25) {};
		\node (2) at (-5.75, -1.25) {};
		\node (3) at (-2.75, 1.25) {};
		\node (4) at (-4.25, 1.25) {};
		\node (14) at (-4.25, -1.25) {};
		\node (17) at (-2.75, -1.25) {};
		\draw [diredge, in=-90, out=90] (1.center) to (3.center);
		\draw [diredge, in=90, out=-90] (4.center) to (2.center);
		\draw [diredge, in=90, out=90, looseness=2.00] (14.center) to (17.center);
\end{pic}
}\hspace{-1pt}}}}

\newcommand{\tinycup}{
\smash{\raisebox{-2pt}{\hspace{-3pt}\ensuremath{\begin{pic}[xscale=0.2,yscale=0.2]
		\node (10) at (3.75, 0.25) {};
		\node (12) at (5.25, 0.25) {};
		\draw [diredge, in=-90, out=-90, looseness=2.00] (10.center) to (12.center);
\end{pic}
}\hspace{-1pt}}}}

\newcommand{\tinycap}{
\smash{\raisebox{-2pt}{\hspace{-3pt}\ensuremath{\begin{pic}[xscale=-0.2,yscale=-0.2]
		\node (10) at (3.75, 0.25) {};
		\node (12) at (5.25, 0.25) {};
		\draw [diredge, in=-90, out=-90, looseness=2.00] (10.center) to (12.center);
\end{pic}
}\hspace{-1pt}}}}

\newcommand{\tinyring}{
\smash{\raisebox{-2pt}{\hspace{-3pt}\ensuremath{\begin{pic}[scale=0.5]
		\draw[diredge=<] (1.4,0.5) to[out=-90,in=0] (1.15,0.25) to[out=180,in=-90] (0.9,0.5) to[out=90,in=180] (1.15,0.75) to[out=0,in=90] (1.4,0.5);
\end{pic}
}\hspace{-1pt}}}}

\newenvironment{pic}[1][]
{\begin{aligned}\begin{tikzpicture}[font=\tiny,#1]}
{\end{tikzpicture}\end{aligned}}

\newcommand{\cat}[1]{\ensuremath{\mathbf{#1}}}

\newcommand{\pure}{\ensuremath{{}^{\text{pure}}}}
\newcommand{\id}[1][]{\ensuremath{\mathrm{id}_{#1}}}
\DeclareMathOperator{\cps}{CP*}
\DeclareMathOperator{\cpm}{CPM}

\theoremstyle{plain}
\newtheorem{theorem}{Theorem}

\newtheorem{corollary}[theorem]{Corollary}
\theoremstyle{definition}
\newtheorem{definition}[theorem]{Definition}

\providecommand{\event}{QPL 2015}
\title{Axiomatizing complete positivity}
\author{Oscar Cunningham\thanks{Supported by EPSRC Studentship OUCL/2014/OAC.} \hspace*{1ex}and Chris Heunen\thanks{Supported by EPSRC Fellowship EP/L002388/1.}
\email{\{oscar.cunningham,heunen\}@cs.ox.ac.uk}
\institute{University of Oxford, Department of Computer Science}}
\def\titlerunning{Axiomatizing complete positivity}
\def\authorrunning{O. Cunningham \& C. Heunen}
\begin{document}
\maketitle
\begin{abstract}
  There are two ways to turn a categorical model for pure quantum theory into one for mixed quantum theory, both resulting in a category of completely positive maps. One has quantum systems as objects, whereas the other also allows classical systems on an equal footing. The former has been axiomatized using environment structures. We extend this axiomatization to the latter by introducing decoherence structures.
\end{abstract}

\section{Introduction}

One of the main draws of categorical quantum mechanics is that it allows models of quantum theory different than that of Hilbert spaces~\cite{abramskycoecke:categoricalsemantics}. Thus one can investigate conceptually which categorical features are responsible for which operational features of quantum theory.
However, this advantage only applies to models of pure quantum theory, compact dagger categories $\cat{C}\pure$.
To model mixed quantum theory requires adding structure on top of $\cat{C}\pure$.
This is accomplished by the following two constructions:
\begin{itemize}
  \item The category $\cpm[\cat{C}\pure]$ has the same objects as $\cat{C}\pure$, but \emph{completely positive maps} between them as morphisms~\cite{selinger:cpm}. Applied to the category of finite-dimensional Hilbert spaces and linear maps, this results in completely positive maps between quantum systems. Hence this category can model dynamics between quantum systems.\footnote{Another construction can add classical systems in a second stage~\cite[Section~5]{selinger:cpm}. That \emph{biproduct completion} is not considered here because it is of a different nature than the $\cpm$ and $\cps$ constructions~\cite{heunenkissingerselinger:idempotents}.}
  \item The category $\cps[\cat{C}\pure]$ extends this by also allowing classical systems~\cite{coeckeheunenkissinger:cpstar}. When applied to the category of finite-dimensional Hilbert spaces and linear maps, this results in completely positive maps between arbitrary finite-dimensional C*-algebras. Hence this category can model arbitrary dynamics, including measurement and controlled preparation of quantum systems.
\end{itemize}
This state of affairs is somewhat unsatisfactory. It would be more in line with the categorical quantum mechanics programme to start off with a category with features that conceptually model mixed quantum theory, rather than having to start off with a category $\cat{C}\pure$ with features that conceptually model pure quantum theory and then bolt more features on top by hand to model mixed quantum theory. That is, we would prefer to \emph{axiomatize} categories modelling mixed quantum theory.
For the CPM construction this has been done~\cite{coecke:cpm,coeckeperdrix:environment}: one can use \emph{environment structures} to see when a given category is of the form $\cpm[\cat{C}\pure]$. 
In this paper, we axiomatize the $\cps$ construction: we introduce \emph{decoherence structures} that allow one to tell when a given category is of the form $\cps[\cat{C}\pure]$. In fact, we show that both $\cpm[\cat{C}\pure]$ and $\cps[\cat{C}\pure]$ satisfy a universal property, paying special attention to the role of \emph{purification} in both~\cite{chiribelladarianoperinotti:derivation}.

The rest of this paper is laid out as follows. Sections~\ref{sec:cpm} and~\ref{sec:cps} discuss the $\cpm$ and $\cps$ constructions and their axiomatization, respectively. Each section first recalls the construction, characterizes the construction in a universal way,  introduces environment/decoherence structures,  and deduces the axiomatization from the universal property.

\paragraph{Future work}

Our axiomatization of $\cps$ raises many interesting questions for future work:
\begin{itemize}
  \item Environment structures have a relatively straightforward physical interpretation, namely discarding the information in a quantum system. The `correct' physical interpretation of decoherence structures is a lot less clear.
  \item Decoherence structures are intimately related to splitting certain idempotents in $\cpm[\cat{C}\pure]$, and relationships to~\cite{selinger:idempotents,heunenkissingerselinger:idempotents} should be investigated. 
  \item Similarly, our axiomatization should give new clues about the seemingly difficult open problem of identifying the Frobenius structures in $\cpm[\cat{C}\pure]$ and $\cps[\cat{C}\pure]$~\cite{heunenboixo:cp,heunenvicarywester:2cp}.
  \item The current work is purely categorical. We expect decoherence structures to give interesting structure in examples such as the category $\cat{Rel}$ of sets and relations.
  \item The axiomatizations of $\cat{C}=\cpm[\cat{C}\pure]$ and $\cat{C}=\cps[\cat{C}\pure]$ both require the user to specify a subcategory $\cat{C}\pure$ of pure morphisms. Ideally this subcategory should be constructed out of $\cat{C}$ itself. There are proposals~\cite{chiribella:pure} that work for the category $\cat{FHilb}$ of finite-dimensional Hilbert spaces, but they do not even lead to a well-defined category in the case of $\cat{Rel}$.
\end{itemize}

\section{Environment structures}\label{sec:cpm}

This section concerns the CPM construction~\cite{selinger:cpm}. We first recall the construction itself and characterizes it in a universal way. We then review environment structures and, by using the universal property of $\cpm$, provide a new~\cite{coecke:cpm,coeckeperdrix:environment,coeckeheunen:infdim} proof that they axiomatize $\cpm$.

We freely make use of compact dagger categories and their graphical calculus~\cite{selinger:graphicallanguages}.
When wires of both types $A$ and $A^*$ arise in one diagram, we will decorate them with arrows in opposite directions. When possible we will suppress coherence isomorphisms in formulae. Finally, recall that $(-)_*$ reverses the order of tensor products, so $f_*$  has type $A^* \to B^* \otimes C^*$ when $f \colon A \to C \otimes B$~\cite{selinger:cpm}.  

\begin{definition}
  If $\cat{C}\pure$ is a compact dagger category, $\cpm[\cat{C}\pure]$ is the compact dagger category where:
  \begin{itemize}
  \item objects of $\cpm[\cat{C}\pure]$ are the same as those of $\cat{C}\pure$;
  \item morphisms $A \to B$ in $\cpm[\cat{C}\pure]$ are morphisms in $\cat{C}\pure$ of the form
    \[\begin{pic}[scale=.75]
      \node[morphism,vflip] (l) at (-.5,0) {$f$};
      \node[morphism] (r) at (.5,0) {$f$};
      \draw[diredge] (l.south) to +(0,-.5) node[below] {$A$};
      \draw[diredge=<] (r.south) to +(0,-.5) node[below] {$A$};
      \draw[diredge=<] (l.north west) to +(0,.5) node[above] {$B$};
      \draw[diredge] (r.north east) to +(0,.5) node[above] {$B$};
      \draw[diredge] (r.north west) to[out=90,in=90] node[above] {$X$} (l.north east);
    \end{pic}\]
    for some object $X$ and morphism $f \colon A \to X \otimes B$ in $\cat{C}\pure$;
  \item identities are inherited from $\cat{C}\pure$, and composition is defined as follows;
    \[\left(\begin{pic}[scale=.75]
      \node[morphism,vflip] (l) at (-.5,0) {$g$};
      \node[morphism] (r) at (.5,0) {$g$};
      \draw[diredge] (l.south) to +(0,-.5) node[below] {$B$};
      \draw[diredge=<] (r.south) to +(0,-.5) node[below] {$B$};
      \draw[diredge=<] (l.north west) to +(0,.5) node[above] {$C$};
      \draw[diredge] (r.north east) to +(0,.5) node[above] {$C$};
      \draw[diredge] (r.north west) to[out=90,in=90] (l.north east);
    \end{pic}\right) \circ 
    \left(\begin{pic}[scale=.75]
      \node[morphism,vflip] (l) at (-.5,0) {$f$};
      \node[morphism] (r) at (.5,0) {$f$};
      \draw[diredge] (l.south) to +(0,-.5) node[below] {$A$};
      \draw[diredge=<] (r.south) to +(0,-.5) node[below] {$A$};
      \draw[diredge=<] (l.north west) to +(0,.5) node[above] {$B$};
      \draw[diredge] (r.north east) to +(0,.5) node[above] {$B$};
      \draw[diredge] (r.north west) to[out=90,in=90] (l.north east);
    \end{pic}\right)
    \quad = \quad
    \begin{pic}[scale=.75]
      \node[morphism,vflip] (l) at (-.75,0) {$f$};
      \node[morphism] (r) at (.75,0) {$f$};
      \node[morphism,vflip] (lt) at (-.75,1) {$g$};
      \node[morphism] (rt) at (.75,1) {$g$};
      \draw[diredge] (l.south) to +(0,-.5) node[below] {$A$};
      \draw[diredge=<] (r.south) to +(0,-.5) node[below] {$A$};
      \draw[diredge=<] (lt.north west) to +(0,.5) node[above] {$C$};
      \draw[diredge] (rt.north east) to +(0,.5) node[above] {$C$};
      \draw[diredge=<] (l.north west) to (lt.south west);
      \draw[diredge] (r.north east) to (rt.south east);
      \draw[diredge] (rt.north west) to[out=90,in=90] (lt.north east);
      \draw[diredge] (r.north west) to[out=120,in=-90] ([xshift=-4mm]rt.north west) to[out=90,in=90] ([xshift=4mm]lt.north east) to[out=-90,in=60] (l.north east);
    \end{pic}
    \]
  \item the tensor unit $I$ and the tensor product of objects are inherited from $\cat{C}\pure$, and the tensor product of morphisms is defined as follows;
    \[\left(\begin{pic}[scale=.75]
      \node[morphism,vflip] (l) at (-.5,0) {$f$};
      \node[morphism] (r) at (.5,0) {$f$};
      \draw[diredge] (l.south) to +(0,-.5) node[below] {$A$};
      \draw[diredge=<] (r.south) to +(0,-.5) node[below] {$A$};
      \draw[diredge=<] (l.north west) to +(0,.5) node[above] {$B$};
      \draw[diredge] (r.north east) to +(0,.5) node[above] {$B$};
      \draw[diredge] (r.north west) to[out=90,in=90] (l.north east);
    \end{pic}\right) \otimes
    \left(\begin{pic}[scale=.75]
      \node[morphism,vflip] (l) at (-.5,0) {$g$};
      \node[morphism] (r) at (.5,0) {$g$};
      \draw[diredge] (l.south) to +(0,-.5) node[below] {$C$};
      \draw[diredge=<] (r.south) to +(0,-.5) node[below] {$C$};
      \draw[diredge=<] (l.north west) to +(0,.5) node[above] {$D$};
      \draw[diredge] (r.north east) to +(0,.5) node[above] {$D$};
      \draw[diredge] (r.north west) to[out=90,in=90] (l.north east);
    \end{pic}\right)
    \quad = \quad
    \begin{pic}[scale=.75]
      \node[morphism,vflip] (l) at (-.5,0) {$g$};
      \node[morphism] (r) at (.5,0) {$g$};
      \draw[diredge] (l.south) to +(0,-.5) node[below] {$C$};
      \draw[diredge=<] (r.south) to +(0,-.5) node[below] {$C$};
      \draw[diredge=<] (l.north west) to +(0,.5) node[above] {$D$};
      \draw[diredge] (r.north east) to +(0,.5) node[above] {$D$};
      \draw[diredge] (r.north west) to[out=90,in=90] (l.north east);
      \node[morphism,vflip] (ll) at (-1.5,0) {$f$};
      \node[morphism] (rr) at (1.5,0) {$f$};
      \draw[diredge] (ll.south) to +(0,-.5) node[below] {$A$};
      \draw[diredge=<] (rr.south) to +(0,-.5) node[below] {$A$};
      \draw[diredge=<] (ll.north west) to +(0,.5) node[above] {$B$};
      \draw[diredge] (rr.north east) to +(0,.5) node[above] {$B$};
      \draw[diredge] (rr.north west) to[out=120,in=60,looseness=.6] (ll.north east);
    \end{pic}
    \]
  \item the dagger is defined as follows.
    \[\left(\begin{pic}[scale=.75]
      \node[morphism,vflip] (l) at (-.5,0) {$f$};
      \node[morphism] (r) at (.5,0) {$f$};
      \draw[diredge] (l.south) to +(0,-.5) node[below] {$A$};
      \draw[diredge=<] (r.south) to +(0,-.5) node[below] {$A$};
      \draw[diredge=<] (l.north west) to +(0,.5) node[above] {$B$};
      \draw[diredge] (r.north east) to +(0,.5) node[above] {$B$};
      \draw[diredge] (r.north west) to[out=90,in=90] (l.north east);
    \end{pic}\right)^\dag
    \quad = \quad
    \begin{pic}[scale=.75]
      \node[morphism,hvflip] (l) at (-.75,0) {$f$};
      \node[morphism,hflip] (r) at (.75,0) {$f$};
      \draw[diredge] (l.south west) to +(0,-.5) node[below] {$B$};
      \draw[diredge=<] (r.south east) to +(0,-.5) node[below] {$B$};
      \draw[diredge=<] (l.north) to +(0,.5) node[above] {$A$};
      \draw[diredge] (r.north) to +(0,.5) node[above] {$A$};
      \draw[diredge] (l.south east) to[out=-90,in=-90,looseness=2] +(.33,0) to +(0,.75) to[out=90,in=90] +(.4,.75) to +(0,-.75) to[out=-90,in=-90,looseness=2] (r.south west);
    \end{pic}
    \]
  \end{itemize}
\end{definition}

There is a canonical functor $P \colon \cat{C}\pure \to \cpm[\cat{C}\pure]$, defined by $P(A)=A$ on objects and $P(f) = f_* \otimes f$ on morphisms. This is a \emph{monoidal dagger functor}: it preserves daggers, and there are a unitary natural transformation $P_2 \colon P(A) \otimes P(B) \to P(A \otimes B)$ and a unitary morphism $P_0 \colon I \to P(I)$ satisfying the appropriate coherence conditions; we will suppress $P_2$ and $P_0$.

Our first main result characterizes $\cpm[\cat{C}\pure]$ (up to monoidal dagger isomorphism) by a universal property. 

\begin{theorem}\label{thm:cpmuniversal}
  For a compact dagger category $\cat{C}\pure$, consider the following category. Objects $(\cat{D},D,\tinyground)$ are categories $\cat{D}$ equipped with a monoidal dagger functor $D \colon \cat{C}\pure \to \cat{D}$ and a morphism $\tinyground \colon D(A) \to I$ for each object $A$ of $\cat{C}\pure$, satisfying:
  \begin{align}
    \begin{pic}
      \node[ground] (1) at (0,.5) {};
      \draw[pure]  (0,0) node[below] {$D(A\otimes B)$} to (1.south);
    \end{pic}
    =
    \begin{pic}
      \node[ground] (1) at (0,.5) {};
      \draw[pure]  (0,0) node[below] {$D(A)$} to (1.south);
      \node[ground] (2) at (0.6,.5) {};
      \draw[pure]  (0.6,0) node[below] {$D(B)$} to (2.south);
    \end{pic}
    \qquad\qquad
    \begin{pic}
      \node[ground] (1) at (0,.5) {};
      \draw[pure]  (0,0) node[below] {$D(I)$} to (1.south);
    \end{pic}
    & =
    \qquad
    \qquad
    \begin{pic}
      \node[ground] (1) at (0,.5) {};
      \draw [pure] (0,0) node[below] {$D(A)$} to (1.south);
    \end{pic}
    =
    \begin{pic}
      \draw (0,0) node[below] {$D(A)$} to (0,.5);
      \draw[diredge] (0,.5) to[out=90,in=90,looseness=1.5] (.6,.5) node[ground,hflip]{};
    \end{pic}
    \label{eq:environment}
    \\
    \begin{pic}[scale=.75]
      \node[morphism,vflip] (l) at (-.5,0) {$f$};
      \node[morphism] (r) at (.5,0) {$f$};
      \draw[diredge] (l.south) to +(0,-.5);
      \draw[diredge=<] (r.south) to +(0,-.5);
      \draw[diredge=<] (l.north west) to +(0,.5);
      \draw[diredge] (r.north east) to +(0,.5);
      \draw[diredge] (r.north west) to[out=90,in=90] (l.north east);
    \end{pic}
    =
    \begin{pic}[scale=.75]
      \node[morphism,vflip] (l) at (-.5,0) {$g$};
      \node[morphism] (r) at (.5,0) {$g$};
      \draw[diredge] (l.south) to +(0,-.5);
      \draw[diredge=<] (r.south) to +(0,-.5);
      \draw[diredge=<] (l.north west) to +(0,.5);
      \draw[diredge] (r.north east) to +(0,.5);
      \draw[diredge] (r.north west) to[out=90,in=90] (l.north east);
    \end{pic}
    \qquad
    \text{in $\cat{C}\pure$}
    \quad
    & \iff
    \quad
    \begin{pic}[scale=.75]
      \node[morphism] (f) at (0,.5) {$D(f)$};
      \draw (f.south) to +(0,-.5);
      \draw ([xshift=3mm]f.north) to +(0,.5);
      \draw ([xshift=-3mm]f.north) to +(0,.2)node[ground,scale=.75]{};
    \end{pic}
    =
    \begin{pic}[scale=.75]
      \node[morphism] (f) at (0,.5) {$D(g)$};
      \draw (f.south) to +(0,-.5);
      \draw ([xshift=3mm]f.north) to +(0,.5);
      \draw ([xshift=-3mm]f.north) to +(0,.2)node[ground,scale=.75]{};
    \end{pic}
    \qquad
    \text{in $\cat{D}$}
  \label{eq:unique}
  \end{align}
  Morphisms $(\cat{D},D,\tinyground) \to (\cat{D'},D',\tinygroundtwo)$ are monoidal dagger functors $F \colon \cat{D} \to \cat{D}'$ such that $F \circ D = D'$ and $F(\tinyground)=\tinygroundtwo$.
  Then:
  \begin{itemize}
  \item $(\cat{D},D,\tinyground)$ is initial in this category if and only if
    \begin{align}
    \text{every morphism of $\cat{D}$ is of the form }
    \begin{pic}[scale=.75]
      \node[morphism] (f) at (0,.5) {$D(f)$};
      \draw (f.south) to +(0,-.5);
      \draw ([xshift=3mm]f.north) to +(0,.5);
      \draw ([xshift=-3mm]f.north) to +(0,.2)node[ground,scale=.75]{};
    \end{pic}.
    \label{eq:purification}
    \end{align}
  \item We may choose $\tinyground$ so that $(\cpm[\cat{C}\pure],P,\tinyground)$ is initial in this category.
  \end{itemize}
\end{theorem}
\noindent
Notice that~\eqref{eq:purification} implies every object of $\cat{D}$ equals $D(A)$ for some $A$ in $\cat{C}\pure$.
\begin{proof}
  We must show that for any $(\cat{D}, D,\tinyground)$ satisfying~\eqref{eq:environment}, \eqref{eq:unique}, and~\eqref{eq:purification}, and any $(\cat{D'}, D',\tinygroundtwo)$ satisfying~\eqref{eq:environment} and~\eqref{eq:unique}, there is a unique monoidal dagger functor $F \colon \cat{D} \to \cat{D'}$ such that $F \circ D=D'$ and $F(\tinyground)=\tinygroundtwo$.

  Every object of $\cat{D}$ is $D(A)$ for some $A$ in $\cat{C}\pure$. Since we need $F \circ D=D'$ we must have $F$ send $D(A)$ to $D'(A)$. On morphisms, $F$ must send $D(f)$ to $D'(f)$ and $\tinyground$ to $\tinygroundtwo$. Therefore we define
  \[
    F \left(
    \begin{pic}[scale=.75]
      \node[morphism] (f) at (0,.5) {$D(f)$};
      \draw (f.south) to +(0,-.5);
      \draw ([xshift=3mm]f.north) to +(0,.5);
      \draw ([xshift=-3mm]f.north) to +(0,.2)node[ground,scale=.75]{};
    \end{pic}
    \right)
    = 
    \begin{pic}[scale=.75]
      \node[morphism] (f) at (0,.5) {$D'(f)$};
      \draw (f.south) to +(0,-.5);
      \draw ([xshift=3mm]f.north) to +(0,.5);
      \draw ([xshift=-3mm]f.north) to +(0,.2)node[groundtwo,scale=.75]{};
    \end{pic}.
  \]
  By~\eqref{eq:purification}, this completely fixes $F$, so it suffices to verify that $F$ is indeed a well-defined monoidal dagger functor. 
  Well-definedness follows from~\eqref{eq:unique}:
  \begin{align*}
    \begin{pic}[scale=.75]
      \node[morphism] (f) at (0,.5) {$D(f)$};
      \draw (f.south) to +(0,-.5);
      \draw ([xshift=3mm]f.north) to +(0,.5);
      \draw ([xshift=-3mm]f.north) to +(0,.2)node[ground,scale=.75]{};
    \end{pic}
    =
    \begin{pic}[scale=.75]
      \node[morphism] (f) at (0,.5) {$D(g)$};
      \draw (f.south) to +(0,-.5);
      \draw ([xshift=3mm]f.north) to +(0,.5);
      \draw ([xshift=-3mm]f.north) to +(0,.2)node[ground,scale=.75]{};
    \end{pic}
    \text{ in }\cat{D}
    \quad & \iff \quad
    \begin{pic}[scale=.75]
      \node[morphism,vflip] (l) at (-.5,0) {$f$};
      \node[morphism] (r) at (.5,0) {$f$};
      \draw[diredge] (l.south) to +(0,-.5);
      \draw[diredge=<] (r.south) to +(0,-.5);
      \draw[diredge=<] (l.north west) to +(0,.5);
      \draw[diredge] (r.north east) to +(0,.5);
      \draw[diredge] (r.north west) to[out=90,in=90] (l.north east);
    \end{pic}
    =
    \begin{pic}[scale=.75]
      \node[morphism,vflip] (l) at (-.5,0) {$g$};
      \node[morphism] (r) at (.5,0) {$g$};
      \draw[diredge] (l.south) to +(0,-.5);
      \draw[diredge=<] (r.south) to +(0,-.5);
      \draw[diredge=<] (l.north west) to +(0,.5);
      \draw[diredge] (r.north east) to +(0,.5);
      \draw[diredge] (r.north west) to[out=90,in=90] (l.north east);
    \end{pic}
    \text{ in }\cat{C}\pure \\
    & \iff \quad
    \begin{pic}[scale=.75]
      \node[morphism] (f) at (0,.5) {$D'(f)$};
      \draw (f.south) to +(0,-.5);
      \draw ([xshift=3mm]f.north) to +(0,.5);
      \draw ([xshift=-3mm]f.north) to +(0,.2)node[groundtwo,scale=.75]{};
    \end{pic}
    =
    \begin{pic}[scale=.75]
      \node[morphism] (f) at (0,.5) {$D'(g)$};
      \draw (f.south) to +(0,-.5);
      \draw ([xshift=3mm]f.north) to +(0,.5);
      \draw ([xshift=-3mm]f.north) to +(0,.2)node[groundtwo,scale=.75]{};
    \end{pic}
    \text{ in }\cat{D'}
  \end{align*}   
  Functoriality of $F$ is established by showing that it preserves identities:
  \[
    \begin{pic}
    \draw[diredge] (0,0) node[below] {$D(A)$} to (0,1);
    \end{pic}
    \quad = \quad
    \begin{pic}
    \node[morphism] (1) at (0,0.5) {$D(\id[A])$};
    \draw[diredge] (0,0) node[below] {$D(A)$} to (1.south);
    \draw[diredge] (1.north) to (0,1);
    \end{pic}
    \quad \stackrel{F}{\longmapsto} \quad
    \begin{pic}
    \node[morphism] (1) at (0,0.5) {$D'(\id[A])$};
    \draw[diredge] (0,0) node[below] {$D'(A)$} to (1.south);
    \draw[diredge] (1.north) to (0,1);
    \end{pic}
    \quad = \quad
    \begin{pic}
    \draw[diredge] (0,0) node[below] {$D'(A)$} to (0,1);
    \end{pic}
  \]
  and that it preserves composition:
  \[
    \begin{pic}
      \node[morphism] (1) at (0,0.5) {$D(f)$};
      \node[morphism] (2) at ([xshift=-5,yshift=25]1.north east) {$D(g)$};
      \draw[diredge] (0,0) to (1.south);
      \draw[diredge] ([xshift=-5]1.north east) to (2.south);
      \draw[diredge] ([xshift=-5]2.north east) to +(0,0.6);    
      \draw ([xshift=5]1.north west) to +(0,0.2) node[ground] {};    
      \draw ([xshift=5]2.north west) to +(0,0.2) node[ground] {};    
    \end{pic}
    \quad = \quad
    \begin{pic}
      \node[morphism] (1) at (0,0.5) {$D(f)$};
      \node[morphism] (2) at ([xshift=-5,yshift=25]1.north east) {$D(g)$};
      \node[ground] (3) at  ([yshift=5]2.north west) {};
      \draw[diredge] (0,0) to (1.south);
      \draw[diredge] ([xshift=-5]1.north east) to (2.south);
      \draw[diredge] ([xshift=-5]2.north east) to +(0,0.6);    
      \draw ([xshift=5]1.north west) to[out=90,in=-90] ([xshift=-3]3);    
      \draw ([xshift=5]2.north west) to[out=90,in=-90] ([xshift=3]3);    
    \end{pic}
    \quad \stackrel{F}{\longmapsto} \quad
    \begin{pic}
      \node[morphism] (1) at (0,0.5) {$D'(f)$};
      \node[morphism] (2) at ([xshift=-5,yshift=25]1.north east) {$D'(g)$};
      \node[groundtwo] (3) at  ([yshift=5]2.north west) {};
      \draw[diredge] (0,0) to (1.south);
      \draw[diredge] ([xshift=-5]1.north east) to (2.south);
      \draw[diredge] ([xshift=-5]2.north east) to +(0,0.6);    
      \draw ([xshift=5]1.north west) to[out=90,in=-90] ([xshift=-3]3);    
      \draw ([xshift=5]2.north west) to[out=90,in=-90] ([xshift=3]3);    
    \end{pic}
    \quad = \quad
    \begin{pic}
      \node[morphism] (1) at (0,0.5) {$D'(f)$};
      \node[morphism] (2) at ([xshift=-5,yshift=25]1.north east) {$D'(g)$};
      \draw[diredge] (0,0) to (1.south);
      \draw[diredge] ([xshift=-5]1.north east) to (2.south);
      \draw[diredge] ([xshift=-5]2.north east) to +(0,0.6);    
      \draw ([xshift=5]1.north west) to +(0,0.2) node[groundtwo] {};    
      \draw ([xshift=5]2.north west) to +(0,0.2) node[groundtwo] {};    
    \end{pic}
  \]
  The functor $F$ is monoidal:
  \[
    \begin{pic}
      \node[morphism] (1) at (0,.5) {$D(f)$};
      \draw[diredge=<] (1.south) to +(0,-.5);
      \draw[diredge] ([xshift=3mm]1.north) to +(0,.5);
      \draw ([xshift=-3mm]1.north) to +(0,.2) node[ground, scale=.7]{};
      \node[morphism] (2) at (1.5,.5) {$D(g)$};
      \draw[diredge=<] (2.south) to +(0,-.5);
      \draw[diredge] ([xshift=3mm]2.north) to +(0,.5);
      \draw ([xshift=-3mm]2.north) to +(0,.2) node[ground, scale=.7]{};
    \end{pic}
    \quad = \quad
    \begin{pic}
      \node[morphism] (1) at (0,.5) {$D(f)$};
      \draw[diredge=<] (1.south) to +(0,-.5);
      \draw[diredge] ([xshift=3mm]1.north) to +(0,.5);
      \node[morphism] (2) at (1.5,.5) {$D(g)$};
      \draw[diredge=<] (2.south) to +(0,-.5);
      \draw[diredge] ([xshift=3mm]2.north) to +(0,.5);
      \node[ground, scale=.7] (g) at ([xshift=-17mm,yshift=2mm]2.north) {};
      \draw ([xshift=-3mm]1.north) to[out=90,in=-90,looseness=.5] ([xshift=-1mm]g.south);
      \draw ([xshift=-3mm]2.north) to[out=90,in=-90,looseness=.5] ([xshift=1mm]g.south);
    \end{pic}
    \quad \stackrel{F}{\longmapsto} \quad
    \begin{pic}
      \node[morphism] (1) at (0,.5) {$D'(f)$};
      \draw[diredge=<] (1.south) to +(0,-.5);
      \draw[diredge] ([xshift=3mm]1.north) to +(0,.5);
      \node[morphism] (2) at (1.5,.5) {$D'(g)$};
      \draw[diredge=<] (2.south) to +(0,-.5);
      \draw[diredge] ([xshift=3mm]2.north) to +(0,.5);
      \node[groundtwo, scale=.7] (g) at ([xshift=-17mm,yshift=3mm]2.north) {};
      \draw ([xshift=-3mm]1.north) to[out=90,in=-90,looseness=.5] ([xshift=-1mm]g.south);
      \draw ([xshift=-3mm]2.north) to[out=90,in=-90,looseness=.5] ([xshift=1mm]g.south);
    \end{pic}
    \quad = \quad
    \begin{pic}
      \node[morphism] (1) at (0,.5) {$D'(f)$};
      \draw[diredge=<] (1.south) to +(0,-.5);
      \draw[diredge] ([xshift=3mm]1.north) to +(0,.5);
      \draw ([xshift=-3mm]1.north) to +(0,.2) node[groundtwo, scale=.7]{};
      \node[morphism] (2) at (1.5,.5) {$D'(g)$};
      \draw[diredge=<] (2.south) to +(0,-.5);
      \draw[diredge] ([xshift=3mm]2.north) to +(0,.5);
      \draw ([xshift=-3mm]2.north) to +(0,.2) node[groundtwo, scale=.7]{};
    \end{pic}
  \]
  Finally, $F$ preserves daggers:
  \[
    \begin{pic}
      \node[morphism,hflip] (1) at (0,0) {$D(f)$};
      \draw[diredge=<] (1.north) to +(0,.5);
      \draw[diredge] ([xshift=3mm]1.south) to +(0,-.5);
      \draw ([xshift=-3mm]1.south) to +(0,-.2) node[ground,scale=.7,hflip]{};
    \end{pic}
    \quad = \quad
    \begin{pic}
      \node[morphism,hflip] (1) at (0,0) {$D(f)$};
      \draw[diredge=<] (1.north) to +(0,.5);
      \draw[diredge] ([xshift=3mm]1.south) to +(0,-.5);
      \draw ([xshift=-3mm]1.south) to[out=-90,in=-90] +(-.5,0) to +(0,.5) node[ground,scale=.7]{};
    \end{pic}
    \quad \stackrel{F}{\longmapsto}
    \begin{pic}
      \node[morphism,hflip] (1) at (0,0) {$D'(f)$};
      \draw[diredge=<] (1.north) to +(0,.5);
      \draw[diredge] ([xshift=3mm]1.south) to +(0,-.5);
      \draw ([xshift=-3mm]1.south) to[out=-90,in=-90] +(-.5,0) to +(0,.5) node[groundtwo,scale=.7]{};
    \end{pic}
    \quad = \quad
    \begin{pic}
      \node[morphism,hflip] (1) at (0,0) {$D'(f)$};
      \draw[diredge=<] (1.north) to +(0,.5);
      \draw[diredge] ([xshift=3mm]1.south) to +(0,-.5);
      \draw ([xshift=-3mm]1.south) to +(0,-.2) node[groundtwo,rotate=180,scale=.7]{};
    \end{pic}
  \] 
  This completes the first part of the proof. It remains to find $\tinyground$ making $(\cpm[\cat{C}\pure],P,\tinyground)$ initial. Taking $\tinyground\colon A^*\otimes A\to I$ to be $\tinycap$ then $(\cpm[\cat{C}\pure],P,\tinyground)$ satisfies \eqref{eq:environment}, \eqref{eq:unique} and \eqref{eq:purification} immediately.
\end{proof}

\begin{definition}
  Let $\cat{C}$ be a compact dagger category, and $\cat{C}\pure$ be a compact dagger subcategory.
  An \emph{environment structure} consists of a morphism $\tinyground \colon A \to I$ for each object $A$ in $\cat{C}\pure$ satisfying
  \begin{align}
    \begin{pic}
      \node[ground] (1) at (0,.5) {};
      \draw[pure]  (0,0) node[below] {$A\otimes B$} to (1.south) ;
    \end{pic}
    =
    \begin{pic}
      \node[ground] (1) at (0,.5) {};
      \draw[pure]  (0,0) node[below] {$A$} to (1.south);
      \node[ground] (2) at (0.6,.5) {};
      \draw[pure]  (0.6,0) node[below] {$B$} to (2.south);
    \end{pic}
    \qquad\qquad
    \begin{pic}
      \node[ground] (1) at (0,.5) {};
      \draw[pure]  (0,0) node[below] {$I$} to (1.south) ;
    \end{pic}~
    &=
    \qquad
    \qquad
    \begin{pic}
      \node[ground] (1) at (0,.5) {};
      \draw (0,0) node[below] {$A$} to (1.south) ;
    \end{pic}
    =
    \begin{pic}
      \draw (0,0) node[below] {$A$} to (0,.5);
      \draw[diredge] (0,.5) to[out=90,in=90,looseness=1.5] (.6,.5) node[ground,hflip]{};
    \end{pic}
  \label{eq:environment2}
  \\
    \begin{pic}[scale=.75]
      \node[morphism,vflip] (l) at (-.5,0) {$f$};
      \node[morphism] (r) at (.5,0) {$f$};
      \draw[diredge] (l.south) to +(0,-.5);
      \draw[diredge=<] (r.south) to +(0,-.5);
      \draw[diredge=<] (l.north west) to +(0,.5);
      \draw[diredge] (r.north east) to +(0,.5);
      \draw[diredge] (r.north west) to[out=90,in=90] (l.north east);
    \end{pic}
    =
    \begin{pic}[scale=.75]
      \node[morphism,vflip] (l) at (-.5,0) {$g$};
      \node[morphism] (r) at (.5,0) {$g$};
      \draw[diredge] (l.south) to +(0,-.5);
      \draw[diredge=<] (r.south) to +(0,-.5);
      \draw[diredge=<] (l.north west) to +(0,.5);
      \draw[diredge] (r.north east) to +(0,.5);
      \draw[diredge] (r.north west) to[out=90,in=90] (l.north east);
    \end{pic}
    \quad &\iff \quad
    \begin{pic}[scale=.75]
      \node[morphism] (f) at (0,.5) {$f$};
      \draw (f.south) to +(0,-.5);
      \draw (f.north east) to +(0,.5);
      \draw (f.north west) to +(0,.2)node[ground,scale=.75]{};
    \end{pic}
    =
    \begin{pic}[scale=.75]
      \node[morphism] (f) at (0,.5) {$g$};
      \draw (f.south) to +(0,-.5);
      \draw (f.north east) to +(0,.5);
      \draw (f.north west) to +(0,.2)node[ground,scale=.75]{};
    \end{pic}
  \label{eq:unique2}
  \end{align}
  for all $f,g\in\cat{C}\pure$. 
  An \emph{environment structure with purification} is an environment structure such that every morphism of $\cat{C}$ is of the form
  
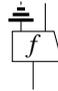
\[
    \begin{pic}[scale=.75]
      \node[morphism] (f) at (0,.5) {$f$};
      \draw (f.south) to +(0,-.5);
      \draw (f.north east) to +(0,.5);
      \draw (f.north west) to +(0,.2)node[ground,scale=.75]{};
    \end{pic}
  \]
  for some $f$ in $\cat C\pure$. 
\end{definition}

Note that environment structures do not require $\cat{C}\pure$ to be \emph{wide} in $\cat{C}$, \textit{i.e.}\ containing every object, unlike~\cite{coeckeperdrix:environment,coeckeheunen:infdim}. However, for environment structures with purification, $\cat{C}\pure$ is necessarily wide in $\cat{C}$.

From the universal property we now deduce that environment structures axiomatize the $\cpm$ construction.

\begin{corollary}
  If a compact dagger category $\cat{C}$ comes with a compact dagger subcategory $\cat{C}\pure$ and an environment structure with purification, there is an isomorphism $\cpm[\cat{C}\pure] \simeq \cat{C}$ of compact dagger categories.
\end{corollary}
\begin{proof}
  Applying Theorem~\ref{thm:cpmuniversal} to the inclusion $\cat{C}\pure \hookrightarrow \cat{C}$ shows that $\cpm[\cat{C}\pure]$ and $\cat{C}$ are both initial, and hence isomorphic.
\end{proof}

\section{Decoherence structures}\label{sec:cps}

This section concerns the $\cps$ construction~\cite{coeckeheunenkissinger:cpstar}. After recalling the construction itself, we characterize it in a universal way. We then introduce decoherence structures and, by using the universal property of $\cps$, prove that decoherence structures axiomatize the $\cps$ construction.

The following definition is not quite the official definition of the $\cps$ construction~\cite{coeckeheunenkissinger:cpstar}, but it is equivalent to it~\cite[Lemma~1.2]{heunenkissingerselinger:idempotents}.

\begin{definition}
  If $\cat{C}\pure$ is a compact dagger category, $\cps[\cat{C}\pure]$ is the compact dagger category given by:
  \begin{itemize}
  \item objects of $\cps[\cat{C}\pure]$ are \emph{special dagger Frobenius structures} in $\cat{C}\pure$: objects $A$ in $\cat{C}\pure$ with morphisms $\tinymult \colon A \otimes A \to A$ and $\tinyunit \colon I \to A$ satisfying: 
  \[ 
    \begin{pic}[scale=.4]
      \node[dot] (t) at (0,1) {};
      \node[dot] (b) at (1,0) {};
      \draw (t) to +(0,1);
      \draw (t) to[out=0,in=90] (b);
      \draw (t) to[out=180,in=90] (-1,0) to (-1,-1);
      \draw (b) to[out=180,in=90] (0,-1);
      \draw (b) to[out=0,in=90] (2,-1);
    \end{pic}
    =
    \begin{pic}[yscale=.4,xscale=-.4]
      \node[dot] (t) at (0,1) {};
      \node[dot] (b) at (1,0) {};
      \draw (t) to +(0,1);
      \draw (t) to[out=0,in=90] (b);
      \draw (t) to[out=180,in=90] (-1,0) to (-1,-1);
      \draw (b) to[out=180,in=90] (0,-1);
      \draw (b) to[out=0,in=90] (2,-1);
    \end{pic}
  \qquad
    \begin{pic}[scale=.4] 
      \node[dot] (d) {};
      \draw (d) to +(0,1);
      \draw (d) to[out=0,in=90] +(1,-1) to +(0,-1);
      \draw (d) to[out=180,in=90] +(-1,-1) node[dot] {};
    \end{pic}
    =
    \begin{pic}[scale=.4]
      \draw (0,0) to (0,3);
    \end{pic}
    =
    \begin{pic}[yscale=.4,xscale=-.4]
      \node[dot] (d) {};
      \draw (d) to +(0,1);
      \draw (d) to[out=0,in=90] +(1,-1) to +(0,-1);
      \draw (d) to[out=180,in=90] +(-1,-1) node[dot] {};
    \end{pic}
  \qquad
    \begin{pic}
      \node[dot] (t) at (0,.6) {};
      \node[dot] (b) at (0,0) {};
      \draw (b.south) to +(0,-0.3);
      \draw (t.north) to +(0,+0.3);
      \draw (b.east) to[out=30,in=-90] +(.15,.3);
      \draw (t.east) to[out=-30,in=90] +(.15,-.3);
      \draw (b.west) to[out=-210,in=-90] +(-.15,.3);
      \draw (t.west) to[out=210,in=90] +(-.15,-.3);
    \end{pic}
    =
    \begin{pic}
      \draw(0,0) to (0,1.5);
    \end{pic}
  \qquad
    \begin{pic}[scale =.6]
      \draw (0,0) to (0,1) to[out=90,in=180] (.5,1.5) to (.5,2);
      \draw (.5,1.5) to[out=0,in=90] (1,1) to[out=-90,in=180] (1.5,.5) to (1.5,0);
      \draw (1.5,.5) to[out=0,in=-90] (2,1) to (2,2);
      \node[dot] at (.5,1.5) {};
      \node[dot] at (1.5,.5) {};
    \end{pic}
    =
    \begin{pic}[scale = 0.6, xscale=-1]
      \draw (0,0) to (0,1) to[out=90,in=180] (.5,1.5) to (.5,2);
      \draw (.5,1.5) to[out=0,in=90] (1,1) to[out=-90,in=180] (1.5,.5) to (1.5,0);
      \draw (1.5,.5) to[out=0,in=-90] (2,1) to (2,2);
      \node[dot] at (.5,1.5) {};
      \node[dot] at (1.5,.5) {};
    \end{pic}
  \]
  we will write 
  $\begin{pic}[scale=.66]
    \node[dot] (d) at (0,0) {};
    \draw[diredge=<] (d) to +(0,-.5);
    \draw[diredge] (d) to[out=0,in=-90] +(.5,.5);
    \draw[diredge=<] (d) to[out=180,in=-90] +(-.5,.5);
  \end{pic}$
  for
  $\begin{pic}[scale=.66]
    \node[dot] (d) at (0,0) {};
    \draw[diredge] (d) to +(0,.5);
    \draw[diredge=<] (d) to[out=0,in=90] +(.4,-.5);
    \draw[diredge=<] (d) to[out=180,in=90] +(-.33,-.33) to[out=-90,in=-90] +(-.33,0) to +(0,.8);
  \end{pic}\,$,
  and
  $\begin{pic}[scale=.66]
    \node[dot] (d) at (0,0) {};
    \draw[diredge] (d) to +(0,.5);
    \draw[diredge=<] (d) to[out=0,in=90] +(.5,-.5);
    \draw[diredge] (d) to[out=180,in=90] +(-.5,-.5);
  \end{pic}$
  for
  $\begin{pic}[scale=.66]
    \node[dot] (d) at (0,0) {};
    \draw[diredge=<] (d) to +(0,-.5);
    \draw[diredge] (d) to[out=0,in=90] +(.4,.5);
    \draw[diredge] (d) to[out=180,in=90] +(-.33,.33) to[out=90,in=90] +(-.33,0) to +(0,-.8);
  \end{pic}\,$;
  \item morphisms $(A,\tinymult,\tinyunit) \to (B,\tinymult[whitedot],\tinyunit[whitedot])$ of $\cps[\cat{C}\pure]$ are morphisms in $\cat{C}\pure$ of the form
    
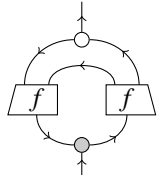
\[
    \begin{pic}[scale=.8]
      \node[morphism,vflip] (l) at (-.75,0) {$f$};
      \node[morphism] (r) at (.75,0) {$f$};
      \draw[diredge] (r.north west) to[out=90,in=90] (l.north east);
      \node[dot] (b) at (0,-.75) {};
      \node[whitedot] (t) at (0,1) {};
      \draw[diredge=<] (b) to +(0,-.5);
      \draw[diredge] (l.south) to[out=-90,in=180] (b.west);
      \draw[diredge] (b.east) to[out=0,in=-90] (r.south);
      \draw[diredge] (r.north east) to[out=90,in=0] (t.east);
      \draw[diredge] (t.west) to[out=180,in=90] (l.north west);
      \draw[diredge] (t.north) to +(0,.5);
    \end{pic}
    \]
    for some object $X$ and morphism $f \colon A \to X \otimes B$ in $\cat{C}\pure$;
  \item identity morphisms and composition in $\cps[\cat{C}\pure]$ are as in $\cat{C}\pure$;
  \item the tensor unit $I$ in $\cps[\cat{C}\pure]$ is the trivial Frobenius structure: tensor unit $I$ in $\cat{C}\pure$, equipped with its coherence isomorphisms $\tinymult \colon I \otimes I \to I$ and $\tinyunit = \id[I] \colon I \to I$;
  \item the tensor product of objects $(A,\tinymult[whitedot],\tinyunit[whitedot])$ and $(B,\tinymult,\tinyunit)$ is $(A \otimes B, \tinymult[whitedot] \hspace*{-1.5mm} \tinymult, \tinyunit[whitedot] \tinyunit)$;
  \item the tensor product of morphisms in $\cps[\cat{C}\pure]$ is inherited from $\cat{C}\pure$;
  \item the dagger in $\cps[\cat{C}\pure]$ is inherited from $\cat{C}\pure$.
  \end{itemize}
\end{definition}

Recall that the \emph{positive-dimensionality} condition~\cite{coeckeheunenkissinger:cpstar} requires precisely that for each object $A$ in $\cat{C}\pure$ there is a positive scalar $z$ such that:
\[
  \begin{pic}
       \draw[diredge] (0,0) to (0,0.5) node[left] {A} to (0,1);
       \node[box] (1) at (0.3,0.3) {z};
       \node[box] (2) at (0.3,0.7) {z};
       \draw[diredge] (1.4,0.5) to[out=-90,in=0] (1.15,0.25) to[out=180,in=-90] (0.9,0.5) node[left] {A} to[out=90,in=180] (1.15,0.75) to[out=0,in=90] (1.4,0.5);
  \end{pic}
  \quad=\quad
  \begin{pic}
       \draw[diredge] (0,0) to (0,0.5) node[right] {A} to (0,1);
  \end{pic}
\]
If $\cat C\pure$ is positive-dimensional, there is a canonical monoidal dagger functor $Q \colon \cat C\pure\to\cps[\cat C\pure]$, defined by $Q(A)=(A\otimes A^*,\hspace{1mm}\tinyz\hspace{0.75mm}\tinypants,\hspace{1mm}\tinyz\hspace{2.5mm}\tinyring\tinycup)$ on objects and $Q(f) = f_* \otimes f$ on morphisms~\cite{coeckeheunenkissinger:cpstar}. In fact, $Q$ factors through the functor $P \colon \cat{C}\pure \to \cpm[\cat{C}\pure]$. We will suppress the coherence morphisms of $Q$.

We are now ready to prove our second main result, that characterizes $\cps[\cat{C}\pure]$ (up to monoidal dagger isomorphism) by a universal property.

\begin{theorem}\label{thm:cpsuniversal}
  For a positive-dimensional compact dagger category $\cat C\pure$, consider the following category. Objects $(\cat D, D,\tinyground,\tinyfrob)$ are categories $\cat D$ equipped with a monoidal dagger functor $D:\cat C\pure \rightarrow \cat D$, a morphism $\tinyground:D(A)\rightarrow I$ for each object $A$ of $\cat C\pure$, and an object $F_\mathcal{A}$ and a morphism $\tinyfrob \colon D(A) \to F_\mathcal{A}$ in $\cat{C}$ for each special dagger Frobenius structure $\mathcal{A}=(A,\tinymult,\tinyunit)$ in $\cat{C}\pure$, satisfying \eqref{eq:environment} and \eqref{eq:unique} as well as:
   \begin{align}
    \begin{pic}
      \node[greytriangle] (1) at (0,.5) {};
      \draw[mixed] (1) to (0,1) node[above] {$ F_{\mathcal{A} \otimes \mathcal{B}}$} ;
      \draw[pure] (1) to (0,0) node[below] {$G(A\otimes B)$};
    \end{pic}
    & =
    \begin{pic}
      \node[whitetriangle] (1) at (0,.5) {};
      \draw[mixed] (1) to (0,1) node[above] {$ F_{\mathcal{A}}$} ;
      \draw[pure] (0,0) node[below] {$D(A)$} to (1) ;
      \node[blacktriangle] (2) at (.5,.5) {};
      \draw[mixed] (2) to (.5,1) node[above] {$ F_{\mathcal{B}}$} ;
      \draw[pure] (.5,0) node[below] {$D(B)$} to (2) ;
    \end{pic}
    &
    \begin{pic}
      \node[whitetriangle] (1) at (0,.5) {};
      \draw[mixed] (1) to (0,1) node[above] {$ F_{\mathcal{I}}$} ; ;
      \draw[pure] (1) to (0,0) node[below] {$D(I)$};
    \end{pic}
    & =
    \qquad
    \label{eq:decoherence}
    \\
    \begin{pic}
      \node[triangle] (1) at (0,.5) {};
      \node[invertedtriangle] (2) at (0,.9) {};
      \draw[mixed,diredge] (1) to (2) ;
      \draw[pure,diredge] (2) to (0,1.4) ;
      \draw[pure,diredge] (0,0) to (1) ;
    \end{pic}
    & =
    \begin{pic}
      \node[dot] (1) at (0,.7) {};
      \node[ground,scale=.75] (2) at (-0.6,1) {};
      \draw[pure,diredge] (1) to (0,1.4) ;
      \draw[pure,diredge] (0,0) to (1) ;
      \draw[pure,diredge] (1) to[out=180,in=-90,looseness=1] (2.south);
    \end{pic}
    & 
    \begin{pic}
      \node[invertedtriangle] (1) at (0,.5) {};
      \node[triangle] (2) at (0,.9) {};
      \draw[mixed,diredge] (2) to (0,1.4) ;
      \draw[mixed,diredge] (0,0) to (1) ;
      \draw[pure,diredge] (1) to (2) ;
    \end{pic}
    & =
    \begin{pic}
     \draw[mixed,diredge] (0,0) to (0,1.4) ;
    \end{pic}
   \label{eq:idempotent}
  \end{align}
  Notice that for~\eqref{eq:decoherence} to hold we must have $F_{\mathcal{A}\otimes\mathcal{B}}=F_  {\mathcal{A}}\otimes F_{\mathcal{B}}$ and $F_{\mathcal{I}}=I$, and that~\eqref{eq:idempotent} abuses the notation $\tinymult$ to mean $D\left(\tinymult\right)$. Morphisms $(\cat D, D,\tinyground,\tinyfrob) \to (\cat D', D',\tinygroundtwo,\tinyfrobtwo)$ are monoidal dagger functors $F\colon \cat D \to \cat D'$ such that $F\circ D=D'$, $F(\tinyground)=\tinygroundtwo$, and $F(\tinyfrob)=\tinyfrobtwo$. Then:
  \begin{itemize}
  \item $(\cat{D},D,\tinyground,\tinyfrob)$ is initial in this category if and only if
   \begin{align}
    \text{every morphism of $\cat{D}$ is of the form }
    \begin{pic}
     \node[pure, morphism] (1) at (0,.75) {$D(f)$};
     \node[greytriangle] (2) at ([xshift=-5,yshift=8]1.north east) {};
     \node[whiteinvertedtriangle] (3) at ([xshift=-5,yshift=-8]1.south east) {};
     \node[ground, scale=0.7] (4) at ([xshift=5,yshift=8]1.north west) {};
     \draw[mixed] (3) to +(0,-.3);
     \draw[mixed] (2) to +(0,.3) ;
     \draw[pure] (3.north) to ([xshift=-5]1.south east);
     \draw[pure] (2.south) to ([xshift=-5]1.north east);
     \draw[pure] (4.south) to ([xshift=5]1.north west);
   \end{pic}.
   \label{eq:purification3}
  \end{align}
  \item We may choose $\tinyground$ and $\tinyfrob$ so that $(\cps[\cat{C}\pure],Q,\tinyground,\tinyfrob)$ is initial in this category.
  \end{itemize}
\end{theorem}
\noindent
Note~\eqref{eq:purification3} implies that every object of $\cat D$ equals $F_{\mathcal{A}}$ for some special dagger Frobenius structure $\mathcal{A}$ in $\cat C\pure$.
\begin{proof}
  We must show that for any $(\cat{D},D,\tinyground,\tinyfrob)$ satisfying \eqref{eq:environment}, \eqref{eq:unique}, \eqref{eq:decoherence}, \eqref{eq:idempotent} and~\eqref{eq:purification3}, and any $(\cat{D}',D',\tinygroundtwo,\tinyfrobtwo)$ satisfying~\eqref{eq:environment}, \eqref{eq:unique}, \eqref{eq:decoherence} and~\eqref{eq:idempotent}, there is a unique monoidal dagger functor $F \colon \cat D \rightarrow \cat D'$ such that $F\circ D=D'$, $D(\tinyground)=\tinygroundtwo$ and $D(\tinyfrob)=\tinyfrobtwo$.

  Every object of $D$ is ${F}_{\mathcal{A}}$ for some special dagger Frobenius structure $\mathcal{A}$ in $\cat C\pure$. Since we need $F\circ D=D'$ we must have $F$ send $F_{\mathcal{A}}$ to ${F}_{\mathcal{A}}'$. On morphisms, $F$ must send $D(f)$ to $D'(f)$,~$\tinyground$ to $\tinygroundtwo$, and $\tinyfrob$ to $\tinyfrobtwo$. Therefore we define
    \[
    F\left(
    \begin{pic}
     \node[pure, morphism] (1) at (0,.75) {$D(f)$};
     \node[greytriangle] (2) at ([xshift=-5,yshift=8]1.north east) {};
     \node[whiteinvertedtriangle] (3) at ([xshift=-5,yshift=-8]1.south east) {};
     \node[ground, scale=0.7] (4) at ([xshift=5,yshift=8]1.north west) {};
     \draw[mixed] (3) to +(0,-.3);
     \draw[mixed] (2) to +(0,.3) ;
     \draw[pure] (3.north) to ([xshift=-5]1.south east);
     \draw[pure] (2.south) to ([xshift=-5]1.north east);
     \draw[pure] (4.south) to ([xshift=5]1.north west);
   \end{pic}
     \right)
    \qquad
    =
    \qquad
    \begin{pic}
     \node[pure, morphism] (1) at (0,.75) {$D'(f)$};
     \node[greytriangletwo] (2) at ([xshift=-5,yshift=8]1.north east) {};
     \node[whiteinvertedtriangletwo] (3) at ([xshift=-5,yshift=-8]1.south east) {};
     \node[groundtwo, scale=0.7] (4) at ([xshift=5,yshift=8]1.north west) {};
     \draw[mixed] (3) to +(0,-.3);
     \draw[mixed] (2) to +(0,.3) ;
     \draw[pure] (3.north) to ([xshift=-5]1.south east);
     \draw[pure] (2.south) to ([xshift=-5]1.north east);
     \draw[pure] (4.south) to ([xshift=5]1.north west);
     \end{pic}.
  \]
  By~\eqref{eq:purification3}, this completely fixes $F$, so it suffices to verify that $F$ is indeed a well-defined monoidal dagger functor. To prove that $F$ is well-defined note that
  \begin{align*}
   &
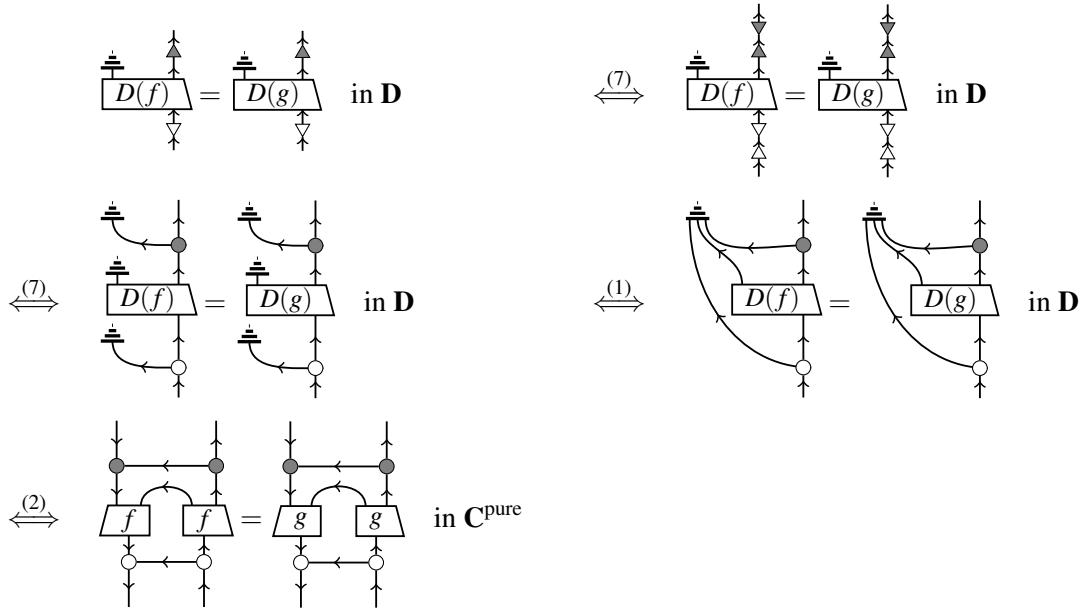
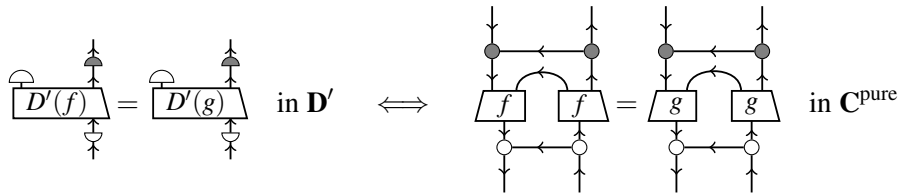
     \begin{pic}
       \node[pure, morphism] (1) at (0,0.9) {$D(f)$};
       \node[darkgreytriangle] (3) at ([xshift=-5,yshift=10]1.north east) {};
       \node[whiteinvertedtriangle] (4) at ([xshift=-5,yshift=-7]1.south east) {};
       \draw[pure] ([xshift=5]1.north west) to +(0,.15) node[ground,scale=0.7]{};
       \draw [pure,diredgeend] (4) to ([xshift=-5]1.south east);
       \draw [pure,diredge] ([xshift=-5]1.north east) to (3) ;
       \draw [pure,diredge=<] (4) to +(0,-.3);
       \draw [pure,diredge] (3) to +(0,.3);       
     \end{pic}
     =
     \begin{pic}
       \node[pure, morphism] (1) at (0,0.9) {$D(g)$};
       \node[darkgreytriangle] (3) at ([xshift=-5,yshift=10]1.north east) {};
       \node[whiteinvertedtriangle] (4) at ([xshift=-5,yshift=-7]1.south east) {};
       \draw[pure] ([xshift=5]1.north west) to +(0,.15) node[ground,scale=0.7]{};
       \draw [pure,diredgeend] (4) to ([xshift=-5]1.south east);
       \draw [pure,diredge] ([xshift=-5]1.north east) to (3) ;
       \draw [pure,diredge=<] (4) to +(0,-.3);
       \draw [pure,diredge] (3) to +(0,.3);       
     \end{pic}
    \quad
     \text{in $\cat{D}$}   
     \quad & \stackrel{\eqref{eq:idempotent}}{\iff} & \quad 
     \begin{pic}
       \node[pure, morphism] (1) at (0,0.9) {$D(f)$};
       \node[darkgreytriangle] (3) at ([xshift=-5,yshift=10]1.north east) {};
       \node[whiteinvertedtriangle] (4) at ([xshift=-5,yshift=-7]1.south east) {};
       \node[darkgreyinvertedtriangle] (5) at ([yshift=10]3) {};  
       \node[whitetriangle] (6) at ([yshift=-10]4) {};
       \draw[pure] ([xshift=5]1.north west) to +(0,.15) node[ground,scale=0.7]{};
       \draw [pure,diredgeend] (4) to ([xshift=-5]1.south east);
       \draw [pure,diredge] ([xshift=-5]1.north east) to (3) ;
       \draw [pure,diredge=<] (6) to +(0,-.3);
       \draw [pure,diredge] (5) to +(0,.3);    
       \draw [mixed,diredge] (3) to (5);
       \draw [mixed,diredge] (6) to (4);
     \end{pic}
     =
     \begin{pic}
       \node[pure, morphism] (1) at (0,0.9) {$D(g)$};
       \node[darkgreytriangle] (3) at ([xshift=-5,yshift=10]1.north east) {};
       \node[whiteinvertedtriangle] (4) at ([xshift=-5,yshift=-7]1.south east) {};
       \node[darkgreyinvertedtriangle] (5) at ([yshift=10]3) {};  
       \node[whitetriangle] (6) at ([yshift=-10]4) {};
       \draw[pure] ([xshift=5]1.north west) to +(0,.15) node[ground,scale=0.7]{};
       \draw [pure,diredgeend] (4) to ([xshift=-5]1.south east);
       \draw [pure,diredge] ([xshift=-5]1.north east) to (3) ;
       \draw [pure,diredge=<] (6) to +(0,-.3);
       \draw [pure,diredge] (5) to +(0,.3);    
       \draw [mixed,diredge] (3) to (5);
       \draw [mixed,diredge] (6) to (4);
     \end{pic}
    \quad
     \text{in $\cat{D}$}   \\ 
     \stackrel{\eqref{eq:idempotent}}{\iff} \quad &
     \begin{pic}
       \node[pure, morphism] (1) at (0,0.9) {$D(f)$};
       \node[darkgreydot] (3) at ([xshift=-5,yshift=15]1.north east) {};
       \node[whitedot] (4) at ([xshift=-5,yshift=-20]1.south east) {};
       \node[ground,scale=0.7] (5) at ([xshift=-25,yshift=10]3) {};
       \node[ground,scale=0.7] (6) at ([xshift=-25,yshift=10]4) {};
       \draw[pure,diredge] (3.east) to[out=180,in=-90,looseness=1] (5);
       \draw[pure,diredge] (4.east) to[out=180,in=-90,looseness=1] (6);
       \draw[pure] ([xshift=5]1.north west) to +(0,.15) node[ground,scale=0.7]{};
       \draw [pure,diredge] (4) to ([xshift=-5]1.south east);
       \draw [pure,diredge] ([xshift=-5]1.north east) to (3) ;
       \draw [pure,diredge=<] (4) to +(0,-.4);
       \draw [pure,diredge] (3) to +(0,.6);
     \end{pic}
     =
     \begin{pic}
       \node[pure, morphism] (1) at (0,0.9) {$D(g)$};
       \node[darkgreydot] (3) at ([xshift=-5,yshift=15]1.north east) {};
       \node[whitedot] (4) at ([xshift=-5,yshift=-20]1.south east) {};
       \node[ground,scale=0.7] (5) at ([xshift=-25,yshift=10]3) {};
       \node[ground,scale=0.7] (6) at ([xshift=-25,yshift=10]4) {};
       \draw[pure,diredge] (3.east) to[out=180,in=-90,looseness=1] (5);
       \draw[pure,diredge] (4.east) to[out=180,in=-90,looseness=1] (6);
       \draw[pure] ([xshift=5]1.north west) to +(0,.15) node[ground,scale=0.7]{};
       \draw [pure,diredge] (4) to ([xshift=-5]1.south east);
       \draw [pure,diredge] ([xshift=-5]1.north east) to (3) ;
       \draw [pure,diredge=<] (4) to +(0,-.4);
       \draw [pure,diredge] (3) to +(0,.6);
     \end{pic}
    \quad
     \text{in $\cat{D}$}   
     \quad & \stackrel{\eqref{eq:environment}}{\iff} & \quad
    \begin{pic}
       \node[pure, morphism] (1) at (0,0.9) {$D(f)$};
       \node[darkgreydot] (3) at ([xshift=-5,yshift=15]1.north east) {};
       \node[whitedot] (4) at ([xshift=-5,yshift=-20]1.south east) {};
       \node[ground,scale=0.7] (5) at ([xshift=-40,yshift=10]3) {};
       \draw[pure,diredge] (3.east) to[out=180,in=-90,looseness=1] ([xshift=3]5);
       \draw[pure,diredge] ([xshift=5]1.north west) to[out=90,in=-90,looseness=1] (5);
       \draw[pure,diredge] (4.east) to[out=180,in=-90,looseness=1] ([xshift=-3]5);
       \draw [pure,diredge] (4) to ([xshift=-5]1.south east);
       \draw [pure,diredge] ([xshift=-5]1.north east) to (3) ;
       \draw [pure,diredge=<] (4) to +(0,-.4);
       \draw [pure,diredge] (3) to +(0,.6);
     \end{pic}
     =
     \begin{pic}
       \node[pure, morphism] (1) at (0,0.9) {$D(g)$};
       \node[darkgreydot] (3) at ([xshift=-5,yshift=15]1.north east) {};
       \node[whitedot] (4) at ([xshift=-5,yshift=-20]1.south east) {};
       \node[ground,scale=0.7] (5) at ([xshift=-40,yshift=10]3) {};
       \draw[pure,diredge] (3.east) to[out=180,in=-90,looseness=1] ([xshift=3]5);
       \draw[pure,diredge] ([xshift=5]1.north west) to[out=90,in=-90,looseness=1] (5);
       \draw[pure,diredge] (4.east) to[out=180,in=-90,looseness=1] ([xshift=-3]5);
       \draw [pure,diredge] (4) to ([xshift=-5]1.south east);
       \draw [pure,diredge] ([xshift=-5]1.north east) to (3) ;
       \draw [pure,diredge=<] (4) to +(0,-.4);
       \draw [pure,diredge] (3) to +(0,.6);
     \end{pic}
    \quad
     \text{in $\cat{D}$}   \\ 
     \stackrel{\eqref{eq:unique}}{\iff} \quad &
     \begin{pic}
       \node[pure, morphism, vflip] (1) at (0,0.9) {$f$};
       \node[pure, morphism] (2) at (1,0.9) {$f$};
       \node[darkgreydot] (3) at ([yshift=15]1.north west) {};
       \node[whitedot] (4) at ([yshift=-10]1.south) {};
       \node[darkgreydot] (5) at ([yshift=15]2.north east) {};
       \node[whitedot] (6) at ([yshift=-10]2.south) {};
       \draw [pure,diredge] (6) to (2.south) ;
       \draw [pure,diredge=<] (4) to (1.south);
       \draw [pure,diredge=<] (4) to (6);
       \draw [pure,diredge=<]  (1.north east) to[out=90,in=90,looseness=1.25] (2.north west) ;
       \draw [pure,diredge=<] (1.north west) to (3) ;
       \draw [pure,diredge] (2.north east) to (5) ;
       \draw [pure,diredge=<] (3) to (5);
       \draw [pure,diredge] (4) to +(0,-.6);
       \draw [pure,diredge=<] (6) to +(0,-.6);
       \draw [pure,diredge=<] (3) to +(0,.6);
       \draw [pure,diredge] (5) to +(0,.6);
     \end{pic}
     =
     \begin{pic}
       \node[pure, morphism, vflip] (1) at (0,0.9) {$g$};
       \node[pure, morphism] (2) at (1,0.9) {$g$};
       \node[darkgreydot] (3) at ([yshift=15]1.north west) {};
       \node[whitedot] (4) at ([yshift=-10]1.south) {};
       \node[darkgreydot] (5) at ([yshift=15]2.north east) {};
       \node[whitedot] (6) at ([yshift=-10]2.south) {};
       \draw [pure,diredge] (6) to (2.south) ;
       \draw [pure,diredge=<] (4) to (1.south);
       \draw [pure,diredge=<] (4) to (6);
       \draw [pure,diredge=<]  (1.north east) to[out=90,in=90,looseness=1.25] (2.north west) ;
       \draw [pure,diredge=<] (1.north west) to (3) ;
       \draw [pure,diredge] (2.north east) to (5) ;
       \draw [pure,diredge=<] (3) to (5);
       \draw [pure,diredge] (4) to +(0,-.6);
       \draw [pure,diredge=<] (6) to +(0,-.6);
       \draw [pure,diredge=<] (3) to +(0,.6);
       \draw [pure,diredge] (5) to +(0,.6);
     \end{pic}
       \quad
     \text{in $\cat{C}\pure$}
  \end{align*}
  Similarly 
  \[
      \begin{pic}
       \node[pure, morphism] (1) at (0,0.9) {$D'(f)$};
       \node[darkgreytriangletwo] (3) at ([xshift=-5,yshift=10]1.north east) {};
       \node[whiteinvertedtriangletwo] (4) at ([xshift=-5,yshift=-7]1.south east) {};
       \draw[pure] ([xshift=5]1.north west) to +(0,.15) node[groundtwo,scale=0.7]{};
       \draw [pure,diredgeend] (4) to ([xshift=-5]1.south east);
       \draw [pure,diredge] ([xshift=-5]1.north east) to (3) ;
       \draw [pure,diredge=<] (4) to +(0,-.3);
       \draw [pure,diredge] (3) to +(0,.3);       
     \end{pic}
     =
     \begin{pic}
       \node[pure, morphism] (1) at (0,0.9) {$D'(g)$};
       \node[darkgreytriangletwo] (3) at ([xshift=-5,yshift=10]1.north east) {};
       \node[whiteinvertedtriangletwo] (4) at ([xshift=-5,yshift=-7]1.south east) {};
       \draw[pure] ([xshift=5]1.north west) to +(0,.15) node[groundtwo,scale=0.7]{};
       \draw [pure,diredgeend] (4) to ([xshift=-5]1.south east);
       \draw [pure,diredge] ([xshift=-5]1.north east) to (3) ;
       \draw [pure,diredge=<] (4) to +(0,-.3);
       \draw [pure,diredge] (3) to +(0,.3);       
     \end{pic}
    \quad
     \text{in $\cat{D'}$}   
     \quad  \iff  \quad
     \begin{pic}
       \node[pure, morphism, vflip] (1) at (0,0.9) {$f$};
       \node[pure, morphism] (2) at (1,0.9) {$f$};
       \node[darkgreydot] (3) at ([yshift=15]1.north west) {};
       \node[whitedot] (4) at ([yshift=-10]1.south) {};
       \node[darkgreydot] (5) at ([yshift=15]2.north east) {};
       \node[whitedot] (6) at ([yshift=-10]2.south) {};
       \draw [pure,diredge] (6) to (2.south) ;
       \draw [pure,diredge=<] (4) to (1.south);
       \draw [pure,diredge=<] (4) to (6);
       \draw [pure,diredge=<]  (1.north east) to[out=90,in=90,looseness=1.25] (2.north west) ;
       \draw [pure,diredge=<] (1.north west) to (3) ;
       \draw [pure,diredge] (2.north east) to (5) ;
       \draw [pure,diredge=<] (3) to (5);
       \draw [pure,diredge] (4) to +(0,-.6);
       \draw [pure,diredge=<] (6) to +(0,-.6);
       \draw [pure,diredge=<] (3) to +(0,.6);
       \draw [pure,diredge] (5) to +(0,.6);
     \end{pic}
     =
     \begin{pic}
       \node[pure, morphism, vflip] (1) at (0,0.9) {$g$};
       \node[pure, morphism] (2) at (1,0.9) {$g$};
       \node[darkgreydot] (3) at ([yshift=15]1.north west) {};
       \node[whitedot] (4) at ([yshift=-10]1.south) {};
       \node[darkgreydot] (5) at ([yshift=15]2.north east) {};
       \node[whitedot] (6) at ([yshift=-10]2.south) {};
       \draw [pure,diredge] (6) to (2.south) ;
       \draw [pure,diredge=<] (4) to (1.south);
       \draw [pure,diredge=<] (4) to (6);
       \draw [pure,diredge=<]  (1.north east) to[out=90,in=90,looseness=1.25] (2.north west) ;
       \draw [pure,diredge=<] (1.north west) to (3) ;
       \draw [pure,diredge] (2.north east) to (5) ;
       \draw [pure,diredge=<] (3) to (5);
       \draw [pure,diredge] (4) to +(0,-.6);
       \draw [pure,diredge=<] (6) to +(0,-.6);
       \draw [pure,diredge=<] (3) to +(0,.6);
       \draw [pure,diredge] (5) to +(0,.6);
     \end{pic}
    \quad
     \text{in $\cat{C}\pure$}   \\ 
\]
  and hence
   \[
     \begin{pic}
       \node[pure, morphism] (1) at (0,0.9) {$D(f)$};
       \node[darkgreytriangle] (3) at ([xshift=-5,yshift=10]1.north east) {};
       \node[whiteinvertedtriangle] (4) at ([xshift=-5,yshift=-7]1.south east) {};
       \draw[pure] ([xshift=5]1.north west) to +(0,.15) node[ground,scale=0.7]{};
       \draw [pure,diredgeend] (4) to ([xshift=-5]1.south east);
       \draw [pure,diredge] ([xshift=-5]1.north east) to (3) ;
       \draw [pure,diredge=<] (4) to +(0,-.3);
       \draw [pure,diredge] (3) to +(0,.3);       
     \end{pic}
     =
     \begin{pic}
       \node[pure, morphism] (1) at (0,0.9) {$D(g)$};
       \node[darkgreytriangle] (3) at ([xshift=-5,yshift=10]1.north east) {};
       \node[whiteinvertedtriangle] (4) at ([xshift=-5,yshift=-7]1.south east) {};
       \draw[pure] ([xshift=5]1.north west) to +(0,.15) node[ground,scale=0.7]{};
       \draw [pure,diredgeend] (4) to ([xshift=-5]1.south east);
       \draw [pure,diredge] ([xshift=-5]1.north east) to (3) ;
       \draw [pure,diredge=<] (4) to +(0,-.3);
       \draw [pure,diredge] (3) to +(0,.3);       
     \end{pic}
    \quad
     \text{in $\cat{D}$}   
     \quad  \iff  \quad
     \begin{pic}
       \node[pure, morphism] (1) at (0,0.9) {$D'(f)$};
       \node[darkgreytriangletwo] (3) at ([xshift=-5,yshift=10]1.north east) {};
       \node[whiteinvertedtriangletwo] (4) at ([xshift=-5,yshift=-7]1.south east) {};
       \draw[pure] ([xshift=5]1.north west) to +(0,.15) node[groundtwo,scale=0.7]{};
       \draw [pure,diredgeend] (4) to ([xshift=-5]1.south east);
       \draw [pure,diredge] ([xshift=-5]1.north east) to (3) ;
       \draw [pure,diredge=<] (4) to +(0,-.3);
       \draw [pure,diredge] (3) to +(0,.3);       
     \end{pic}
     =
     \begin{pic}
       \node[pure, morphism] (1) at (0,0.9) {$D'(g)$};
       \node[darkgreytriangletwo] (3) at ([xshift=-5,yshift=10]1.north east) {};
       \node[whiteinvertedtriangletwo] (4) at ([xshift=-5,yshift=-7]1.south east) {};
       \draw[pure] ([xshift=5]1.north west) to +(0,.15) node[groundtwo,scale=0.7]{};
       \draw [pure,diredgeend] (4) to ([xshift=-5]1.south east);
       \draw [pure,diredge] ([xshift=-5]1.north east) to (3) ;
       \draw [pure,diredge=<] (4) to +(0,-.3);
       \draw [pure,diredge] (3) to +(0,.3);       
     \end{pic}
    \quad
     \text{in $\cat{D'}$.}\\ 
  \]
  Functoriality of $F$ is established by showing that it preserves identities:
    \[
    \begin{pic}
    \draw[diredge] (0,0) node[below] {$F_\mathcal{A}$} to (0,1.8);
    \end{pic}
    \quad = \quad
    \begin{pic}
       \node[pure, morphism] (1) at (0,0.9) {$D(id_A)$};
       \node[greytriangle] (3) at ([yshift=10]1.north) {};
       \node[greyinvertedtriangle] (4) at ([yshift=-7]1.south) {};
       \draw [pure,diredgeend] (4) to (1.south);
       \draw [pure,diredge] (1.north) to (3) ;
       \draw [pure,diredge=<] (4) to (0,0) node[below] {$F_\mathcal{A}$};
       \draw [pure,diredge] (3) to (0,1.8);       
     \end{pic}
    \quad\stackrel{F}\longmapsto\quad
    \begin{pic}
       \node[pure, morphism] (1) at (0,0.9) {$D'(id_A)$};
       \node[greytriangletwo] (3) at ([yshift=10]1.north) {};
       \node[greyinvertedtriangletwo] (4) at ([yshift=-7]1.south) {};
       \draw [pure,diredgeend] (4) to (1.south);
       \draw [pure,diredge] (1.north) to (3) ;
       \draw [pure,diredge=<] (4) to (0,0) node[below] {$F_\mathcal{A}$};
       \draw [pure,diredge] (3) to (0,1.8);       
     \end{pic}
    \quad=\quad
    \begin{pic}
    \draw[diredge] (0,0) node[below] {$F'_\mathcal{A}$} to (0,1.8);
    \end{pic}
    \]
    and that it preserves composition:
    \[
    \begin{pic}
   \node[morphism] (1) at (0,0.8) {$D(g)$};
   \node[whiteinvertedtriangle] (4) at ([xshift=-5,yshift=-10]1.south east) {};
   \node[greytriangle] (3) at ([xshift=-5,yshift=8]1.north east) {};
   \node[greyinvertedtriangle] (8) at ([yshift=20]3) {};   
   \node[morphism] (5) at ([yshift=47]1) {$D(f)$};
   \node[blacktriangle] (7) at ([xshift=-5,yshift=7]5.north east) {};
   \node[ground,scale=0.8] (2) at ([xshift=5,yshift=7]5.north west){};
   \node[ground,scale=0.8] (6) at ([xshift=5,yshift=7]1.north west){};
   \draw [mixed,diredge] ([yshift=-10]4.center) to (4.center) ;
   \draw [pure,diredge] (4) to[out=90,in=-90] ([xshift=-5]1.south east);
   \draw [pure,diredge] ([xshift=-5]1.north east) to (3) ;
   \draw [mixed,diredge] (3) to (8) ;
   \draw [pure,diredgeend] (8) to[out=90, in=-90]  ([xshift=-5]5.south east);
   \draw [pure,diredgeend] ([xshift=-5]5.north east) to (7) ;
   \draw [mixed,diredgeend] (7) to +(0,.4);
   \draw [pure,diredgeend] ([xshift=5]5.north west) to[out=90,in=-90] (2) ;
   \draw [pure,diredgeend] ([xshift=5]1.north west) to[out=90,in=-90] (6) ;
  \end{pic}
  \quad=\quad
    \begin{pic}
   \node[morphism] (1) at (0,0.8) {$D(g)$};
   \node[whiteinvertedtriangle] (4) at ([xshift=-5,yshift=-10]1.south east) {};
   \node[greydot] (3) at ([xshift=-5,yshift=18]1.north east) {};
   \node[morphism] (5) at ([yshift=47]1) {$D(f)$};
   \node[blacktriangle] (7) at ([xshift=-5,yshift=7]5.north east) {};
   \node[ground,scale=0.8] (2) at ([xshift=5,yshift=7]5.north west){};
   \node[ground,scale=0.8] (6) at ([xshift=5,yshift=7]1.north west){};
   \node[ground,scale=0.8] (9) at ([xshift=5,yshift=27]1.north west){};
   \draw [mixed,diredge] ([yshift=-10]4.center) to (4.center) ;
   \draw [pure,diredge] (4) to[out=90,in=-90] ([xshift=-5]1.south east);
   \draw [pure,diredge] ([xshift=-5]1.north east) to (3) ;
   \draw [pure,diredgeend] (3) to[out=90, in=-90]  ([xshift=-5]5.south east);
   \draw [pure,diredgeend] ([xshift=-5]5.north east) to (7) ;
   \draw [mixed,diredgeend] (7) to +(0,.4);
   \draw [pure,diredgeend] ([xshift=5]5.north west) to[out=90,in=-90] (2) ;
   \draw [pure,diredgeend] ([xshift=5]1.north west) to[out=90,in=-90] (6) ;
   \draw [pure,diredge] (3) to[out=180, in=-90]  (9);
  \end{pic}
    \quad=\quad
    \begin{pic}
   \node[morphism] (1) at (0,0.8) {$D(g)$};
   \node[whiteinvertedtriangle] (4) at ([xshift=-5,yshift=-10]1.south east) {};
   \node[greydot] (3) at ([xshift=-5,yshift=18]1.north east) {};
   \node[morphism] (5) at ([yshift=47]1) {$D(f)$};
   \node[blacktriangle] (7) at ([xshift=-5,yshift=7]5.north east) {};
   \node[ground,scale=0.8] (2) at ([xshift=-10,yshift=7]5.north west){};
   \draw [mixed,diredge] ([yshift=-10]4.center) to (4.center) ;
   \draw [pure,diredge] (4) to[out=90,in=-90] ([xshift=-5]1.south east);
   \draw [pure,diredge] ([xshift=-5]1.north east) to (3) ;
   \draw [pure,diredgeend] (3) to[out=90, in=-90]  ([xshift=-5]5.south east);
   \draw [pure,diredgeend] ([xshift=-5]5.north east) to (7) ;
   \draw [mixed,diredgeend] (7) to +(0,.4);
   \draw [pure,diredge] ([xshift=5]1.north west) to[out=90,in=-90] ([xshift=-3]2) ;
   \draw [pure,diredge] (3) to[out=180, in=-90]  (2);
   \draw [pure,diredge] ([xshift=5]5.north west) to[out=90,in=-90] ([xshift=3]2) ;
  \end{pic}
   \quad \stackrel{F}\longmapsto\quad
    \begin{pic}
   \node[morphism] (1) at (0,0.8) {$D'(g)$};
   \node[whiteinvertedtriangletwo] (4) at ([xshift=-5,yshift=-10]1.south east) {};
   \node[greydot] (3) at ([xshift=-5,yshift=18]1.north east) {};
   \node[morphism] (5) at ([yshift=47]1) {$D'(f)$};
   \node[blacktriangletwo] (7) at ([xshift=-5,yshift=7]5.north east) {};
   \node[groundtwo,scale=0.8] (2) at ([xshift=-10,yshift=7]5.north west){};
   \draw [mixed,diredge] ([yshift=-10]4.center) to (4.center) ;
   \draw [pure,diredge] (4) to ([xshift=-5]1.south east);
   \draw [pure,diredge] ([xshift=-5]1.north east) to (3) ;
   \draw [pure,diredgeend] (3) to[out=90, in=-90]  ([xshift=-5]5.south east);
   \draw [pure,diredgeend] ([xshift=-5]5.north east) to (7) ;
   \draw [mixed,diredgeend] (7) to +(0,.4);
   \draw [pure,diredge] ([xshift=5]1.north west) to[out=90,in=-90] ([xshift=-3]2) ;
   \draw [pure,diredge] (3) to[out=180, in=-90]  (2);
   \draw [pure,diredge] ([xshift=5]5.north west) to[out=90,in=-90] ([xshift=3]2) ;
  \end{pic}
   \quad =\quad
    \begin{pic}
   \node[morphism] (1) at (0,0.8) {$D'(g)$};
   \node[whiteinvertedtriangletwo] (4) at ([xshift=-5,yshift=-10]1.south east) {};
   \node[greydot] (3) at ([xshift=-5,yshift=18]1.north east) {};
   \node[morphism] (5) at ([yshift=47]1) {$D'(f)$};
   \node[blacktriangletwo] (7) at ([xshift=-5,yshift=7]5.north east) {};
   \node[groundtwo,scale=0.8] (2) at ([xshift=5,yshift=7]5.north west){};
   \node[groundtwo,scale=0.8] (6) at ([xshift=5,yshift=7]1.north west){};
   \node[groundtwo,scale=0.8] (9) at ([xshift=5,yshift=27]1.north west){};
   \draw [mixed,diredge] ([yshift=-10]4.center) to (4.center) ;
   \draw [pure,diredge] (4) to ([xshift=-5]1.south east);
   \draw [pure,diredge] ([xshift=-5]1.north east) to (3) ;
   \draw [pure,diredgeend] (3) to[out=90, in=-90]  ([xshift=-5]5.south east);
   \draw [pure,diredgeend] ([xshift=-5]5.north east) to (7) ;
   \draw [mixed,diredgeend] (7) to +(0,.4);
   \draw [pure,diredgeend] ([xshift=5]5.north west) to[out=90,in=-90] (2) ;
   \draw [pure,diredgeend] ([xshift=5]1.north west) to[out=90,in=-90] (6) ;
   \draw [pure,diredge] (3) to[out=180, in=-90]  (9);
  \end{pic}
  \quad=\quad
  \begin{pic}
   \node[morphism] (1) at (0,0.8) {$D'(g)$};
   \node[whiteinvertedtriangletwo] (4) at ([xshift=-5,yshift=-10]1.south east) {};
   \node[greytriangletwo] (3) at ([xshift=-5,yshift=8]1.north east) {};
   \node[greyinvertedtriangletwo] (8) at ([yshift=20]3) {};   
   \node[morphism] (5) at ([yshift=47]1) {$D'(f)$};
   \node[blacktriangletwo] (7) at ([xshift=-5,yshift=7]5.north east) {};
   \node[groundtwo,scale=0.8] (2) at ([xshift=5,yshift=7]5.north west){};
   \node[groundtwo,scale=0.8] (6) at ([xshift=5,yshift=7]1.north west){};
   \draw [mixed,diredge] ([yshift=-10]4.center) to (4.center) ;
   \draw [pure,diredge] (4) to ([xshift=-5]1.south east);
   \draw [pure,diredge] ([xshift=-5]1.north east) to (3) ;
   \draw [mixed,diredge] (3) to (8) ;
   \draw [pure,diredgeend] (8) to  ([xshift=-5]5.south east);
   \draw [pure,diredgeend] ([xshift=-5]5.north east) to (7) ;
   \draw [mixed,diredgeend] (7) to +(0,.4);
   \draw [pure,diredgeend] ([xshift=5]5.north west) to[out=90,in=-90] (2) ;
   \draw [pure,diredgeend] ([xshift=5]1.north west) to[out=90,in=-90] (6) ;
  \end{pic}
    \]
    The functor $F$ is monoidal:
    \[
    \begin{pic}
       \node[pure, morphism] (1) at (0,0.9) {$D(f)$};
       \node[darkgreytriangle] (3) at ([xshift=-5,yshift=20]1.north east) {};
       \node[whiteinvertedtriangle] (4) at ([xshift=-5,yshift=-17]1.south east) {};
       \node[pure, morphism] (2) at (1.5,0.9) {$D(g)$};
       \node[darkgreytriangle] (5) at ([xshift=-5,yshift=20]2.north east) {};
       \node[whiteinvertedtriangle] (6) at ([xshift=-5,yshift=-17]2.south east) {};
       \draw[pure] ([xshift=5]1.north west) to +(0,.15) node[ground,scale=0.7]{};
       \draw [pure,diredgeend] (4) to ([xshift=-5]1.south east);
       \draw [pure,diredge] ([xshift=-5]1.north east) to (3) ;
       \draw [pure,diredge=<] (4) to +(0,-.6);
       \draw [pure,diredge] (3) to +(0,.6); 
       \draw[pure] ([xshift=5]2.north west) to +(0,.15) node[ground,scale=0.7]{};
       \draw [pure,diredgeend] (6) to ([xshift=-5]2.south east);
       \draw [pure,diredge] ([xshift=-5]2.north east) to (5) ;
       \draw [pure,diredge=<] (6) to +(0,-.6);
       \draw [pure,diredge] (5) to +(0,.6);             
     \end{pic}
    \quad=\quad
    \begin{pic}
       \node[pure, morphism] (1) at (0,0.9) {$D(f)$};
       \node[pure, morphism] (2) at (1.5,0.9) {$D(g)$};
       \node[darkgreytriangle] (5) at ([xshift=-5,yshift=20]2.north east) {};
       \node[whiteinvertedtriangle] (6) at ([xshift=-5,yshift=-17]2.south east) {};
       \node[ground] (7) at ([xshift=8,yshift=20]1.north west){};
       \draw[pure,diredge] ([xshift=5]1.north west) to[out=90,in=-90] ([xshift=-3]7);
       \draw[pure,diredge] ([xshift=5]2.north west) to[out=90,in=-90] ([xshift=3]7);
       \draw [pure,diredge] ([xshift=-5]1.north east) to[out=90,in=-90] ([xshift=-1]5);
       \draw [pure,diredge] ([xshift=-4]2.north east) to[out=90,in=-90] ([xshift=1]5) ;
       \draw [pure,diredge] ([xshift=-1]5.center) to[out=90,in=-90]  ([xshift=-5,yshift=36]1.north east); 
       \draw [pure,diredge] ([xshift=1]5.center) to[out=90,in=-90] +(0,.6);
       \draw [pure,diredge] (6.center) to[out=90,in=-90] ([xshift=-5]1.south east);
       \draw [pure,diredge=<] ([xshift=-1]6.center) to[out=-90,in=90] ([xshift=-5,yshift=-36]1.south east);
       \draw [pure,diredge=<] ([xshift=-4]2.south east) to[out=-90,in=90] ([xshift=1]6.center);
       \draw [pure,diredge=<] ([xshift=1]6.center) to[out=-90,in=90] +(0,-.6);
     \end{pic}
     \quad\stackrel{F}\longmapsto\quad
      \begin{pic}
       \node[pure, morphism] (1) at (0,0.9) {$D'(f)$};
       \node[pure, morphism] (2) at (1.5,0.9) {$D'(g)$};
       \node[darkgreytriangletwo] (5) at ([xshift=-5,yshift=20]2.north east) {};
       \node[whiteinvertedtriangletwo] (6) at ([xshift=-5,yshift=-17]2.south east) {};
       \node[groundtwo] (7) at ([xshift=8,yshift=20]1.north west){};
       \draw[pure,diredge] ([xshift=5]1.north west) to[out=90,in=-90] ([xshift=-3]7);
       \draw[pure,diredge] ([xshift=5]2.north west) to[out=90,in=-90] ([xshift=3]7);
       \draw [pure,diredge] ([xshift=-5]1.north east) to[out=90,in=-90] ([xshift=-1]5) ;
       \draw [pure,diredge=<]  ([xshift=1]5) to[out=-90,in=90] ([xshift=-5]2.north east) ;
       \draw [pure,diredge] ([xshift=-1]5.center) to[out=90,in=-90]  ([xshift=-5,yshift=36]1.north east);
       \draw [pure,diredge] ([xshift=1]5.center) to[out=90,in=-90] +(0,.6);
       \draw [pure,diredge] (6.center) to[out=90,in=-90] ([xshift=-5]1.south east);
       \draw [pure,diredge=<] ([xshift=-1]6.center) to[out=-90,in=90] ([xshift=-5,yshift=-36]1.south east);
       \draw [pure,diredge=<]  ([xshift=-4]2.south east) to[out=-90,in=90] ([xshift=1]6.center);
       \draw [pure,diredge=<] ([xshift=1]6.center) to[out=-90,in=90] +(0,-.6);
     \end{pic}
   \quad =\quad
    \begin{pic}
       \node[pure, morphism] (1) at (0,0.9) {$D'(f)$};
       \node[darkgreytriangletwo] (3) at ([xshift=-5,yshift=20]1.north east) {};
       \node[whiteinvertedtriangletwo] (4) at ([xshift=-5,yshift=-17]1.south east) {};
       \node[pure, morphism] (2) at (1.5,0.9) {$D'(g)$};
       \node[darkgreytriangletwo] (5) at ([xshift=-5,yshift=20]2.north east) {};
       \node[whiteinvertedtriangletwo] (6) at ([xshift=-5,yshift=-17]2.south east) {};
       \draw[pure] ([xshift=5]1.north west) to +(0,.15) node[groundtwo,scale=0.7]{};
       \draw [pure,diredgeend] (4) to ([xshift=-5]1.south east);
       \draw [pure,diredge] ([xshift=-5]1.north east) to (3) ;
       \draw [pure,diredge=<] (4) to +(0,-.6);
       \draw [pure,diredge] (3) to +(0,.6); 
       \draw[pure] ([xshift=5]2.north west) to +(0,.15) node[groundtwo,scale=0.7]{};
       \draw [pure,diredgeend] (6) to ([xshift=-5]2.south east);
       \draw [pure,diredge] ([xshift=-5]2.north east) to (5) ;
       \draw [pure,diredge=<] (6) to +(0,-.6);
       \draw [pure,diredge] (5) to +(0,.6);             
     \end{pic}
    \]
Finally, $F$ preserves daggers:
    \[
    \begin{pic}
      \node[morphism, hflip] (1) at (0,.5) {$D(f)$};
      \node[whitetriangle] (3) at ([xshift=-5,yshift=10]1.north east) {};
      \node[blackinvertedtriangle] (4) at ([xshift=-5,yshift=-10]1.south east) {};
      \node[ground, scale=-0.7] (2) at ([yshift=-4,xshift=5]1.south west) {};
      \draw[diredge] ([xshift=5]1.south west) to (2.north) ;
      \draw[diredge]  ([xshift=-5]1.south east) to (4);
      \draw[diredge]  (4) to +(0,-.3);
      \draw[diredge=<]  ([xshift=-5]1.north east) to (3);
      \draw[diredge=<]  (3) to +(0,0.3);
    \end{pic}
    \quad=\quad
    \begin{pic}
      \node[morphism, hflip] (1) at (0,.5) {$D(f)$};
      \node[whitetriangle] (3) at ([xshift=-5,yshift=10]1.north east) {};
      \node[blackinvertedtriangle] (4) at ([xshift=-5,yshift=-10]1.south east) {};
      \node[ground, scale=0.7] (2) at ([yshift=15,xshift=-8]1.south west) {};
      \draw[diredge] ([xshift=5]1.south west) to[out=-90,in=-90,looseness=2] (2) ;
      \draw[diredge]  ([xshift=-5]1.south east) to (4);
      \draw[diredge]  (4) to +(0,-.3);
      \draw[diredge=<]  ([xshift=-5]1.north east) to (3);
      \draw[diredge=<]  (3) to +(0,0.3);
    \end{pic}
   \quad  \stackrel{F}\longmapsto\quad
      \begin{pic}
      \node[morphism, hflip] (1) at (0,.5) {$D'(f)$};
      \node[whitetriangletwo] (3) at ([xshift=-5,yshift=10]1.north east) {};
      \node[blackinvertedtriangletwo] (4) at ([xshift=-5,yshift=-10]1.south east) {};
      \node[groundtwo, scale=0.7] (2) at ([yshift=15,xshift=-8]1.south west) {};
      \draw[diredge] ([xshift=5]1.south west) to[out=-90,in=-90,looseness=2] (2) ;
      \draw[diredge]  ([xshift=-5]1.south east) to (4);
      \draw[diredge]  (4) to +(0,-.3);
      \draw[diredge=<]  ([xshift=-5]1.north east) to (3);
      \draw[diredge=<]  (3) to +(0,0.3);
    \end{pic}
    \quad=\quad
   \begin{pic}
      \node[morphism, hflip] (1) at (0,.5) {$D'(f)$};
      \node[whitetriangletwo] (3) at ([xshift=-5,yshift=10]1.north east) {};
      \node[blackinvertedtriangletwo] (4) at ([xshift=-5,yshift=-10]1.south east) {};
      \node[groundtwo, scale=-0.7] (2) at ([yshift=-4,xshift=5]1.south west) {};
      \draw[diredge] ([xshift=5]1.south west) to (2.north) ;
      \draw[diredge]  ([xshift=-5]1.south east) to (4);
      \draw[diredge]  (4) to +(0,-.3);
      \draw[diredge=<]  ([xshift=-5]1.north east) to (3);
      \draw[diredge=<]  (3) to +(0,0.3);
    \end{pic}
    \]
    This completes the first part of the proof. It remains to define $\tinyground$ and $\tinyfrob$ so that $(\cps[\cat{C}\pure],Q,\tinyground,\tinyfrob)$ satisfies \eqref{eq:environment}, \eqref{eq:unique}, \eqref{eq:decoherence}, \eqref{eq:idempotent} and \eqref{eq:purification3}. We again take $\tinyground\colon A^*\otimes A\to I$ to be $\tinycap$. Note that this is indeed a morphism of $\cps[\cat C\pure]$ since it can be written in the required form:
$$
\begin{pic}
       \node (4) at (0.3, 0) {};
       \node (3) at (0.7, 0) {};
       \node (9) at (0, 0.5) {};
       \node (10) at (0.3, 0.5) {};
       \node (11) at (0.7, 0.5) {};
       \node (8) at (1, 0.5) {};
       \node (2) at (0, 1) {};
       \node (14) at (0.3, 1) {};
       \node (17) at (0.7, 1) {};
       \node (1) at (1,1) {};
       \node[box] (5) at (-0.2,0.15) {z};
       \node[box] (6) at (-0.4,0.7) {$\phi$};
       \node[box] (7) at (1.4,0.7) {$\phi$};
       \draw [diredge=<, in=90, out=-90] (8.center) to (3.center);
       \draw [diredge=<, in=-90, out=90] (4.center) to (9.center);
       \draw [diredge=<] (1.center) to (8.center);
       \draw [diredge=<] (9.center) to (2.center);
       \draw [diredge=<] (14.center) to (10.center);
       \draw [diredge=<] (11.center) to (17.center);
       \draw [diredge=<, in=-90, out=-90, looseness=2.00] (10.center) to (11.center);
       \draw [diredge, in=90, out=90, looseness=2.00] (14.center) to (17.center);
       \draw [diredge, in=90, out=90, looseness=2.00] (1.center) to (2.center);
       \draw [diredge=<, in=90, out=90, looseness=2.00] (6.north) to (7.north);
\end{pic}
$$
where $\phi$ has been chosen such that $\phi^\dagger\circ\phi=z$. Such a $\phi$ exists because $z$ is positive. Note that \eqref{eq:environment} and \eqref{eq:unique} are satisfied as before. Now take $F_\mathcal{A}=\mathcal{A}$ and let $\tinyfrob\colon Q(A)\to\mathcal{A}$ be given by 
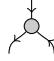
\[
\begin{pic}
       \node (1) at (0,0) {}; 
       \node (3) at (0.6,0) {}; 
       \node (5) at (0.3,1) {}; 
       \node[greydot] (7) at (0.3,0.5) {};
       \draw[diredge, looseness=1.25] (7) to[out=-135,in=90] (1);
       \draw[diredge, looseness=1.25] (3) to[out=90,in=-45] (7);
       \draw[diredge=<, looseness=1.25] (7) to[out=90,in=-90] (5);
\end{pic}
\]
(which is a morphism of $\cps[\cat C\pure]$ for similar reasons as for $\tinyground$). Then \eqref{eq:decoherence} holds by definition of tensor product of Frobenius structures, and \eqref{eq:idempotent} also holds by the spider theorem for special Frobenius structures~\cite[Lemma~3.1]{coeckepaquette:naimark}:
\[
\begin{pic}
       \node (1) at (0,0) {}; 
       \node (3) at (0.6,0) {}; 
       \node (5)[greydot] at (0.3,1) {}; 
       \node[greydot] (7) at (0.3,0.5) {};
       \node (2) at (0,1.5) {}; 
       \node (4) at (0.6,1.5) {}; 
       \draw[diredge, looseness=1.25] (7) to[out=-135,in=90] (1);
       \draw[diredge, looseness=1.25] (3) to[out=90,in=-45] (7);
       \draw[diredge=<, looseness=1.25] (7) to[out=90,in=-90] (5);
       \draw[diredge=<, looseness=1.25] (5) to[out=135,in=-90] (2);
       \draw[diredge, looseness=1.25] (5) to[out=45,in=-90] (4);
\end{pic}
=
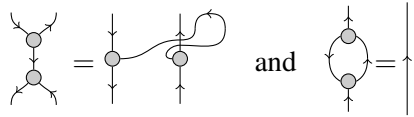
\begin{pic}
       \node[greydot] (1) at (-0.5,0.6) {};
       \node[greydot] (2) at (0.4,0.6) {};
       \draw[diredge=<] (-0.5,0) to (1);
       \draw[diredge=<] (1) to (-0.5,1.2);
       \draw[diredge] (0.4,0) to (2);
       \draw[diredge] (2) to (0.4,1.2);
       \draw[looseness=1.25] (1) to[out=0,in=180] (0.4,0.85) to[out=0,in=-90] (0.6,1);
       \draw[diredge=<,looseness=1.25] (0.6,1) to[out=90,in=180] (0.8,1.2) to[out=0,in=90] (1,1);
       \draw[looseness=1.25] (1,1) to[out=-90,in=0] (0.4,0.78) to[out=180,in=180,looseness=2] (2);
\end{pic}
\quad
\text{and}
\quad
\begin{pic}
      \node[dot] (t) at (0,.6) {};
      \node[dot] (b) at (0,0) {};
      \draw[diredge=<] (b.south) to +(0,-0.3);
      \draw[diredge] (t.north) to +(0,+0.3);
      \draw[diredgeend] (b.east) to[out=30,in=-90] +(.15,.3);
      \draw (t.east) to[out=-30,in=90] +(.15,-.3);
      \draw (b.west) to[out=-210,in=-90] +(-.15,.3);
      \draw[diredgeend] (t.west) to[out=210,in=90] +(-.15,-.3);
    \end{pic}
    =
    \begin{pic}
      \draw[diredge] (0,0) to (0,1.5);
    \end{pic}
\]

Finally (\ref{eq:purification3}) holds because it is precisely the requirement that every morphism is of the form
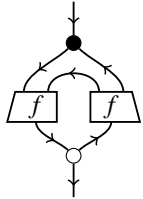
\[
    \begin{pic}
      \node[pure,morphism, vflip] (1) at (0,0.9) {$f$};
      \node[pure,morphism] (2) at (1,0.9) {$f$};
      \node[blackdot] (3) at (0.5,1.75) {};
      \node[whitedot] (4) at (0.5,.25) {};
      \draw [pure,diredge=<] (0.5,-.3) to (4.south) ;
      \draw [pure,diredge=<] (4) to[out=135,in=-90,looseness=0.75] (1.south);
      \draw [pure,diredge] (4) to[out=45,in=-90,looseness=0.75] (2.south);
      \draw [pure,diredge=<]  (1.north east) to[out=90,in=90,looseness=1.25] (2.north west) ;
      \draw [pure,diredge=<] (1.north west) to[out=90,in=-135,looseness=0.75] (3) ;
      \draw [pure,diredge] (2.north east) to[out=90,in=-45,looseness=0.75] (3) ;
      \draw [pure,diredge=<] (3) to (0.5,2.3);
    \end{pic}
  \]
  and is hence vacuously satisfied.
\end{proof}

\begin{definition}
  Let $\cat{C}$ be a compact dagger category, and $\cat{C}\pure$ a compact dagger subcategory.
  A \emph{decoherence structure} is an environment structure together with an object $F_\mathcal{A}$ and a morphism $\tinyfrob \colon A \to F_\mathcal{A}$ in $\cat{C}$ for each special dagger Frobenius structure $\mathcal{A}=(A,\tinymult,\tinyunit)$ in $\cat{C}\pure$, satisfying:
  \begin{align}
    \begin{pic}
      \node[greytriangle] (1) at (0,.5) {};
      \draw[mixed] (1) to (0,1) node[above] {$ F_{\mathcal{A} \otimes \mathcal{B}}$} ;
      \draw[pure] (1) to (0,0) node[below] {$A\otimes B$};
    \end{pic}
    & =
    \begin{pic}
      \node[whitetriangle] (1) at (0,.5) {};
      \draw[mixed] (1) to (0,1) node[above] {$ F_{\mathcal{A}}$} ;
      \draw[pure] (0,0) node[below] {$A$} to (1) ;
      \node[darkgreytriangle] (2) at (.5,.5) {};
      \draw[mixed] (2) to (.5,1) node[above] {$ F_{\mathcal{B}}$} ;
      \draw[pure] (.5,0) node[below] {$B$} to (2) ;
    \end{pic}
    &
    \begin{pic}
      \node[whitetriangle] (1) at (0,.5) {};
      \draw[mixed] (1) to (0,1) node[above] {$ F_{\mathcal{I}}$} ; ;
      \draw[pure] (1) to (0,0) node[below] {$I$};
    \end{pic}
    & =
    \qquad
    \label{eq:decoherence2}
    \\
    \begin{pic}
      \node[triangle] (1) at (0,.5) {};
      \node[invertedtriangle] (2) at (0,.9) {};
      \draw[mixed,diredge] (1) to (2) ;
      \draw[pure,diredge] (2) to (0,1.4) ;
      \draw[pure,diredge] (0,0) to (1) ;
    \end{pic}
    & =
    \begin{pic}
      \node[dot] (1) at (0,.7) {};
      \node[ground,scale=.75] (2) at (-0.6,1) {};
      \draw[pure,diredge] (1) to (0,1.4) ;
      \draw[pure,diredge] (0,0) to (1) ;
      \draw[pure,diredge] (1) to[out=180,in=-90,looseness=1] (2.south);
    \end{pic}
    & 
    \begin{pic}
      \node[invertedtriangle] (1) at (0,.5) {};
      \node[triangle] (2) at (0,.9) {};
      \draw[mixed,diredge] (2) to (0,1.4) ;
      \draw[mixed,diredge] (0,0) to (1) ;
      \draw[pure,diredge] (1) to (2) ;
    \end{pic}
    & =
   \begin{pic}
     \draw[mixed,diredge] (0,0) to (0,1.4) ;
   \end{pic}
   \label{eq:idempotent2}
  \end{align}
  A \emph{decoherence structure with purification} is a decoherence structure such that every morphism of $\cat{C}$ is of the form 
  
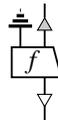
\[
   \begin{pic}
     \node[pure, morphism] (1) at (0,.75) {$f$};
     \node[greytriangle] (2) at ([yshift=8]1.north east) {};
     \node[whiteinvertedtriangle] (3) at ([yshift=-8]1.south east) {};
     \node[ground, scale=0.7] (4) at ([yshift=8]1.north west) {};
     \draw[mixed] (3) to +(0,-.3);
     \draw[mixed] (2) to +(0,.3) ;
     \draw[pure] (3.north) to (1.south east);
     \draw[pure] (2.south) to (1.north east);
     \draw[pure] (4.south) to (1.north west);
   \end{pic}
  \]
  for some $f$ in $\cat{C}\pure$.
\end{definition}

Note that~\eqref{eq:decoherence2} entails $F_{\mathcal{A} \otimes \mathcal{B}}= F_{\mathcal{A}}\otimes  F_{\mathcal{B}}$ and $ F_{\mathcal{I}}=I$. Note also that in a decoherence structure with purification, each object of $\cat C$ must be of the form $F_\mathcal{A}$  for a special dagger Frobenius structure $\mathcal{A}$ in $\cat C\pure$.

From the universal property we now deduce that decoherence structures axiomatize the $\cps$ construction.

\begin{corollary}
  If a compact dagger category $\cat{C}$ comes with a compact dagger subcategory $\cat{C}\pure$ and a decoherence structure with purification, then there is an isomorphism $\cps[\cat{C}\pure] \simeq \cat{C}$ of compact dagger categories.
\end{corollary}
\begin{proof}
  Applying Theorem~\ref{thm:cpsuniversal} to the inclusion $\cat{C}\pure \hookrightarrow \cat{C}$ shows that $\cps[\cat{C}\pure]$ and $\cat{C}$ are both initial, and hence isomorphic.
\end{proof}

\bibliographystyle{eptcs}
\bibliography{cpstaraxioms}

\end{document}